%% file: SJSC_Draft_V6.tex
\documentclass[sjsc]{siamart1116}

\usepackage{xcolor}
\usepackage{amsmath}
\usepackage{amssymb}
\usepackage{lipsum}
\usepackage{amsfonts}
\usepackage{graphicx}
\usepackage{subfigure}
\usepackage{epstopdf}
\usepackage{algorithmic}
\usepackage[english]{babel}
\usepackage[T1]{fontenc}
\usepackage[ansinew]{inputenc}
\usepackage{lmodern}	% font definition
\usepackage{tikz}
\usetikzlibrary{decorations.pathmorphing} % noisy shapes
\usetikzlibrary{fit}					% fitting shapes to coordinates
\usetikzlibrary{backgrounds}	% drawing the background after the foreground
\usepackage{amsopn}
\usepackage{multirow}
\usepackage{adjustbox}

\numberwithin{theorem}{section}
\numberwithin{equation}{section}

\newcommand{\TheTitle}{Compressing Large-Scale Wave Propagation Models via Phase-Preconditioned Rational Krylov Subspaces}

\newcommand{\ShortTheTitle}{Phase-preconditioned Rational Krylov subspaces} %Abbreviated title max 50 char
\hypersetup{colorlinks={false}}

\newcommand{\TheAuthors}{V.~Druskin, R.~Remis, M.~Zaslavsky, and J.~Zimmerling}

\headers{\ShortTheTitle}{\TheAuthors}

\title{{\TheTitle}\thanks{Submitted to the editors DATE.
\funding{This research is supported by the Dutch Technology Foundation STW (project number 14222), which is part of the Netherlands Organisation for Scientific Research (NWO), and which is partly funded by the Ministry of Economic Affairs.}}}

\author{
  Vladimir~Druskin\thanks{Schlumberger-Doll Research, Cambridge, MA 19104-2688
    (\email{druskin@slb.com},\email{mzaslavsky@slb.com}).} %, \url{http://www.imag.com/\string~ddoe/}).}
  \and
  Rob~F.~Remis\thanks{Circuits and Systems Group, Faculty of Electrical Engineering, Mathematics and Computer
Science, Delft University of Technology, Mekelweg 4, 2628 CD Delft, The Netherlands 
(\email{R.F.Remis@tudelft.nl},\email{j.t.zimmerling@tudelft.nl}).}
  \and
    Mikhail~Zaslavsky\footnotemark[2]
    \and
     J\"orn~T.~Zimmerling\footnotemark[3].
}

\ifpdf
\hypersetup{
  pdftitle={\TheTitle},
  pdfauthor={\TheAuthors}
}
\fi

\include{preamble} % definitions

\begin{document}

\maketitle

% REQUIRED
\begin{abstract}
Rational Krylov subspace (RKS) techniques are well-established and powerful tools for projection-based model reduction of time-invariant dynamic systems. For hyperbolic wavefield problems, such techniques perform well in configurations where only a few modes contribute to the field. RKS methods, however, are fundamentally limited by the Nyquist-Shannon sampling rate, making them unsuitable for the approximation of wavefields in configuration characterized by large travel times and propagation distances, since wavefield responses in such configurations are highly oscillatory in the frequency-domain. To overcome this limitation, we propose to precondition the RKSs by factoring out the rapidly varying frequency-domain field oscillations. The remaining amplitude functions are generally slowly varying functions of source position and spatial coordinate and allow for a significant compression of the approximation subspace.  Our one-dimensional analysis together with numerical experiments for large scale 2D acoustic models show superior approximation properties of  preconditioned RKS compared with the standard RKS model-order reduction. The preconditioned RKS results in a reduction of the frequency sampling well below the Nyquist-Shannon rate, a weak dependence of the RKS size on the number of inputs and outputs for multiple-input/multiple-output (MIMO) problems, and, most importantly, in a significant coarsening of the finite-difference grid used to generate the RKS.  A prototype implementation indicates that the preconditioned RKS algorithm is competitive in the modern high performance computing environment.
\end{abstract}

% REQUIRED

% REQUIRED

\section{Introduction}%Field relevance
Numerical modeling of wave propagation is fundamental to many applications in design optimization and wavefield imaging. In the oil and gas industry, for instance, the solution of the Maxwell equations is required to invert electromagnetic measurements, while in seismic imaging the solution to the elastodynamic wave equation is needed to ultimately image the subsurface of the Earth. 

Finite difference discretization of the governing wave equations leads to large-scale linear systems, whose solution is computationally intense. Imaging and optimization often use multiple frequencies, sources, and receivers, which leads to systems that need to be evaluated for multiple right-hand sides, time-steps or frequencies, depending on whether the problem is solved in the time- or frequency-domain. Therefore, these so-called multiple-input/multiple-output (MIMO) systems have a high demand on memory and computational power, causing long runtimes. To be more specific, let us consider a surface seismic imaging problem in a $k$ dimensional space ($1\le k\le 3$), with maximal propagation distance of $N$ wavelengths. This would require the solution of a discretized system with $O(N^k)$ state variables,
$O(N^{k-1})$ sources and receivers, and $O(N)$ frequencies or time steps \cite{Mulder1}. Model-order reduction aims to reduce the complexity and computational burden of large-scale problems and here we target all three of these factors.

Recently, promising results were obtained in the time-domain via multiscale model reduction \cite{chung2014generalized,druskin2016multi}. The time-domain multiscale algorithms can be efficiently parallelized via domain-decomposition, but time stepping still needs to be {\color{black} carried out} sequentially, while frequency-domain problems can be solved in parallel for different frequencies. Here we consider interpolatory projection-based model reduction in the frequency-domain, e.g., see \cite{InterpolatoryModelReduction}. The essence of this approach is the projection of the underlying system onto a rational Krylov subspace (originally introduced by Ruhe for eigenvalue computations \cite{Ruhe1984391}), which produces good quality, low order approximations if the spectrum of the system is well separated from the frequency interval of interest, as in the case of diffusion PDEs, e.g., \cite{beckermann2009error, Druskin20123883, DruskinDiffusion, guttel2010rational}.
In the context of wavefield modeling, such a separation of the clustered eigenvalues is introduced by losses present in the media or by the use of absorbing boundary conditions for the truncation of unbounded domains. Projection-based reduced-order models (ROMs) for wavefield problems may therefore exhibit fast convergence, especially for resonant configurations with few isolated resonant eigenmodes \cite{Bindel2006,Remis2013PKS,Remis2014EKS}. Some modifications of the RKS projection method can also be competitive for problems with smooth initial conditions leading to effective suppression of highly oscillatory eigenmodes \cite{gockler2013convergence}. 

{\color{black} Usually, the computational cost of projection-based reduced-order modeling is dominated by the generation of a suitable projection basis, e.g., see \cite{ParametricMOR, ROM_projectionref}. In the interpolatory projection-based ROM  the Helmholtz equation has to be solved at different frequencies (shifts) and the span of these  solutions forms the RKS basis}. The solution obtained from the Galerkin projection onto this subspace interpolates at the shifts, which are therefore also known as interpolation points. Moreover, for coinciding sources and receivers the transfer function and its first derivative is interpolated at these points. A general drawback of an RKS approach is that the number of interpolation points can become  large when wave field solutions with large travel times or propagation distances are of interest. Such wavefields are highly oscillatory in the frequency-domain and the Nyquist-Shannon sampling theorem states that this oscillatory field should be sampled with at least one point per oscillation (two points per wavelength). Consequently, the number of interpolation points required to accurately represent the wavefield increases with {\color{black} the} propagation distance. Moreover, discretization grids in Helmholtz solvers must also resolve wavefield oscillations. This requirement has an even more dramatic effect on the computational cost due to poor scalability of the available solvers. In favorable situations the best sampling rates approaching the Nyquist limit can be achieved with high-order spectral methods and their outgrowths. However, their cost per unknown can be significantly higher compared with less accurate low-order methods due to loss of sparsity.
In this paper, we show that the sampling demand can be significantly lowered by adding phase information to the model-order reduction technique leading to phase-preconditioned RKSs (PPRKS). 
Preconditioning of Krylov subspaces for model reduction is a tough and still open problem in general. However, to achieve it for particular applications one can try to incorporate the underlying physics and asymptotic analysis to arrive at PPRKS. 
Our approach is related to other known approaches in the field of oscillatory wave problem computation, such as  preconditioners for 
Helmholtz solvers \cite{engquist2011sweeping, Haber20114403}, Filon quadrature \cite{iserles2006highly}, and a recent approach to data compression using phase-tracking \cite{li2015phase}. 

In particular, we construct RKSs using polar decompositions of frequency-\emph{dependent} basis functions. These {\color{black} decompositions} consist of a product of {\it smooth} amplitude functions and a known frequency-dependent {\it oscillatory} phase term. The phase term is determined from high-frequency asymptotic expansions such as the WKB (Wentzel--Kramers--Brillouin) approximation \cite{Bender&Orszag}. The amplitude functions are computed by splitting the RKS into incoming and outgoing waves (by applying one-way wave operators) and factoring out the corresponding phase terms. {\color{black} Analogous to Filon quadrature,} we handle the highly oscillatory phase functions analytically and the smooth amplitude function numerically. By developing a block version of phase-preconditioned RKS for MIMO problems, we are also able to factor out the main dependence of the RKS on the input (source) location. This feature, and the reduction of the number of interpolation points mentioned above, leads to a significant compression of the approximation space.

Finally, the resulting phase-preconditioned ROMs can also extrapolate to frequencies outside the interval of interpolation points, since the basis functions are frequency-dependent and the amplitude functions are smooth for smoothly varying wave speed profiles. This enables us to coarsen the second-order finite-difference grid used for the RKS generation. 

{\color{black} In conclusion,} with phase-preconditioned RKS we can effectively reduce all of the above mentioned factors contributing to the complexity of the MIMO wavefield problem. The overall goal is to approximate the transfer functions from multiple sources to multiple receivers with a small reduced-order model that honors the physics of the underlying wave equation. The approach uses a coarse grid and low frequency interpolation points to build an RKS and to obtain smooth amplitude functions. Using high-frequency asymptotic expansions, this RKS is extrapolated to high frequencies and evaluated on a fine grid. The projection of a fine grid wave operator onto the extrapolated RKS gauges the ROM to the fine scale we intend to model. In this way, fine scale wave scattering and large scale wave propagation can be combined, which allows us to obtain a ROM valid for all time scales. RKS algorithms for wavefield problems are at a disadvantage compared {\color{black} with polynomial and extended Krylov subspace algorithms} when it comes to computational memory consumption as   the basis needs to be saved for RKS methods. The compression of the approximation space to a {\color{black} small number of} amplitudes and phases, however, leads to a reduction in the computational memory demand of the proposed method. 

In section~2, we start with a short discussion on the wave equation and formulate the wavefield problem of interest for a single-input/single output (SISO) configuration. Subsequently, we introduce a standard RKS in section~3 and 
construct field approximations in the frequency-domain. We show that this RKS approach is structure-preserving and that the transfer function of reduced-order models based on this RKS is a Hermite interpolant of the transfer function for a coinciding source-receiver pair. In section~4, we take the RKS approach of section~3 as a starting point and introduce the phase-preconditioned RKS for one-dimensional SISO configurations. {\color{black} We show that phase preconditioning is structure-preserving and retains the interpolation properties of standard RKS Galerkin projection.} The main result of this section is that for a piecewise constant wave speed profile, the new method yields the exact solution with the number of interpolation points equal to the number of homogeneous layers, i.e., this number plays the same role as the problem dimensionality in a conventional RKS approach. Section~5 discusses the algorithm for higher spatial dimensions in a MIMO setting using a block version of phase-preconditioned RKS. Finite-difference implementation via a two-grid algorithm is discussed in section~6. In section~7 we illustrate the performance of the proposed RKS techniques through a number of two-dimensional numerical experiments. Section~8 discusses the implementation of the proposed method on parallel computation architectures and the conclusions can be found in section~9. Throughout this manuscript, quantities in the time-domain are denoted in upright font, while quantities in the Laplace domain are written in italic.

\section{Problem formulation}
In this paper we address the problem of solving the Green's function for wave equations within a spectral interval of interest. We start the discussion by considering the scalar, isotropic, continuous wave equation on $\mathbb{R}^{k} \times [0, \infty[$
\be
\Delta {\rm u} - \frac{1}{\nu^2}{\rm u}_{tt} =- \frac{1}{\nu^2}\delta(t)\delta(x-x_\text{S}) , \quad {\rm u}|_{t=0}=0,\, {\rm u}_t|_{t=0}={\color{black}0.} \label{eq:contWEQ}
\ee 
In this equation, $\Delta$ denotes the $k$-dimensionalLaplace operator and the position vector is $x \in\RR^k$ ($1\le k\le 3$). {\color{black} Furthermore, $\nu( x)>0$} is a  wave speed distribution  in $L^\infty [\mathbb{R}^k] $, and ${\rm u}( x,t)$ is the wavefield with a compact support for all finite times.    

After Laplace transformation, equation \cref{eq:contWEQ} becomes
\be\label{eq:LaplaceWEQ}
\Delta u - \frac{s^2}{\nu^2} {u}=- \frac{1}{\nu^2}\delta (x-x_\text{S}),
\ee
where $s$ is the complex Laplace parameter with $\Re(s) \geq 0$. The Laplace domain wavefield $u$ satisfies the limiting absorption principle, i.e. $u$ vanishes at infinity for $\Re(s)> 0$ and converges to the solution of  Helmholtz's equation that satisfies the outgoing radiation condition as the Laplace parameter $s$ approaches the imaginary axis via the right-half of the complex $s$-plane. 

Let $\Omega$ be a bounded subdomain  of $\mathbb{R}^k$ such that $x_{\text{S}}\in \Omega$. We now equivalently reduce  the original  problem  on the unbounded domain to a problem on $\Omega$ by   
 considering equation~(\ref{eq:LaplaceWEQ}) in the weak formulation and testing this equation with a testing function~$p$. This gives
\be
\int_{\Omega} \overline{p}\left( \Delta - \frac{s^2 }{\nu^2} \right){u} \, {\text d} x =  -\frac{1}{\nu(x_\text{S})^2} \overline{p}(x_\text{S}),
\ee
where the overbar denotes complex conjugation. After integration by parts, we obtain
\be\label{eq:weakFormHelmholtzIntegral}
-\int_{\Omega} ( \nabla \overline{p}) \cdot (\nabla {u}) \, {\text d} x -  \int_{\Omega} \overline{p}\frac{s^2 }{\nu^2} { u} \, {\text d} x + \int_{\partial \Omega} \overline{p} \frac{\partial{u}}{\partial n}  \, {\text d} x =- \frac{1}{\nu(x_\text{S})^2} \overline{p}(x_\text{S}),
\ee
with $ \frac{\partial{{u}}}{\partial n}$ the derivative of ${u}$ in the direction of the outward-pointing normal on $\partial \Omega$. Finally, introducing the Dirichlet-to-Neumann (DtN) map $D(s)$ on $\partial \Omega$ such that $ \frac{\partial{{u}}}{\partial n} = D(s) {u}$, the above equation can be written as 
\be\label{eq:weakFormHelmholtzIntegral2}
-\int_{\Omega} (\nabla \overline{p})\cdot (\nabla {u}) \, {\text d} x -  \int_{\Omega} \overline{p}\frac{s^2 }{\nu^2} { u}  \,{\text d} x + \int_{\partial \Omega} \overline{p} D(s) {u}  \,{\text d} x = -\frac{1}{\nu(x_\text{S})^2} \overline{p}(x_\text{S}).
\ee
Without the boundary integral  (third term on the left-hand side of the above equation) this equation is linear is $s^2$; the DtN map, however, is a nonlinear function of frequency $s$ \cite{PML2016SiamReview}.\\

\noindent {\it Notation:}
To better draw similarities between continuous and discrete formulations, we will treat the  complex-valued functions $u$ and $p$ as vectors from $\RR^\infty$ in our linear algebraic derivations and  introduce the inner product
\be% We never use the bilinear form so it is deleted
p^H u=\int_{\Omega} \overline{p} {u} \, {\text d}x.
\ee
We note, that  $u$ and $p$ for $k>1$ have singularities at $x_\text{S}$ that may make this inner product divergent. To avoid this, we assume by default that instead of $\delta(x-x_\text{S})$ we have some regular approximation of the delta function.  After discretization, $u$ and $p$ become finite-dimensional vectors from $\RR^{N}$ and the issue of diverging integrals due to singularities disappears.
 In this notation, superscript $H$ denotes the Hermitian transpose for vectors and an inner product with complex conjugation for functions. 
Operators are printed with capital italic letters like $A$ and for linear combinations such as 
\[
q_m=\alpha_{1} g^{[1]} + \alpha_2 g^{[2]} + ... + \alpha_m g^{[m]},
\]
with coefficients $\alpha_i$ and expansion functions $g^{[i]}$, we write $q_{m}=G_{m} \bsz$ with $\bsz=[\alpha_1,...,\alpha_m]^T$ and the expansion functions are stored as columns in the function array~$G_{m}$, i.e.,  $G_{m}\in \RR^{\infty\times m}$ (Sometimes called quasimatrix \cite{Townsend20140585}).  Finally, finite-dimensional matrices are printed using a capital sans serif font (like $\bsA$).\\

\noindent Using the notation outlined above, we now introduce the wave operator $Q(s)$ to rewrite equation \cref{eq:weakFormHelmholtzIntegral2} as
\be\label{eq:weakFormHelmholtzVector}
 p^H Q(s) u = - \frac{1}{\nu(x_\text{S})^2} \overline{p}(x_\text{S}).
\ee
First, we note  that real and imaginary parts of $Q(s)$ are self-adjoint.
In the time-domain, the wavefield is obviously real-valued and consequently operator ${Q}(s)$ and the field ${u}(s)$ satisfy the  Schwarz reflection principle
\be\label{eq:symmetryAstreched}
{Q}(\overline{s})= \overline{{Q}(s)} \text{ and } {u}(\overline{s})= \overline{{u}(s)},
\ee
from which it immediately follows that the spectrum of ${Q}(s)$ is symmetric under complex conjugation. 

Global energy conservation for problem (\ref{eq:contWEQ}) leads to passivity of ${Q}(s)$, which can be defined via its nonlinear numerical range ${\cal W} \left\{ Q(s) \right\}$  (also known as nonlinear field of values), e.g., see \cite{barnett1983matrix},\cite{Druskin20123883}. Specifically, for a nonlinear operator-valued function $A(s)$, the nonlinear numerical range ${\cal W} \left\{ A(s) \right\}$ (e.g., see\cite{barnett1983matrix},\cite{Druskin20123883}) is defined as
\be
{\cal W} \left\{ A(s) \right\}= \left\{s\in \mathbb{C} : x^H A(s)x=0 \quad \forall x\in \mathbb{C}^k\backslash{{0}}  \right\}.
\ee
Passivity of dynamic system (\ref{eq:weakFormHelmholtzVector}) is equivalent to the condition
\begin{equation}
\label{passivity}\Re {\cal W} \left\{ Q(s) \right\} \le 0.
\end{equation}
In the following sections we discuss a reduced-order modeling technique that preserves the above mentioned symmetry properties, the Schwarz reflection principle, and passivity.

\section{Structure preserving rational Krylov subspace reduction}
As a first step towards an efficient rational Krylov methodology for multi-frequency wavefield problems, we construct field approximations or reduced-order models based on an interpolatory rational Krylov subspace containing single frequency solutions (snapshots) of the problem as trial and testing space. Specifically, our approach is to define an RKS of order $m$ as
\be
{\calK}^{m}(\kappa) = {\rm span} \left\{ {u}(s_1), {u}(s_2), \dots, {u}(s_m) \right\}
\ee
with $m$ distinct shifts $\kappa=\left[ s_1, \dots, s_m\right]$ and to use its real form, the RKS
\be
{\calK}_\text{R}^{2 m}(\kappa) = {\rm span} \left\{ \Re {\calK}^{m}(\kappa), \Im {\calK}^{m}(\kappa) \right\},
\ee
as an test and trial space. The real and imaginary parts of the snapshots $u(s_i)$ spanning ${\cal K}_\text{R}^{2m}(\kappa)$ are always linearly independent, since the eigenfunction expansion of the Dirac distribution appearing on the right-hand side of equation~(\ref{eq:LaplaceWEQ}) has an infinite number of terms. Furthermore, from the symmetry given in \cref{eq:symmetryAstreched} it follows that ${\calK}^{m}(\kappa) \subset {\cal K}_\text{R}^{2m}(\kappa)$ and ${\calK}^{m}(\overline{\kappa}) \subset {\cal K}_\text{R}^{2m}(\kappa)$ and a projection onto the subspace ${\cal K}_\text{R}^{2m}(\kappa)$ will therefore preserve the Schwarz reflection principle leading to real-valued, time-domain wavefield approximations. In the following subsections we will construct the reduced-order wavefield approximations, discuss their structure, and show the interpolation properties of these approximations.

\subsection{Reduced-order solution}
We start by approximating the weak solution of equation \cref{eq:weakFormHelmholtzVector} by an element from the space ${\cal K}_R^{2m}(\kappa)$. To this end, let the functions $v^{[1]}$, $v^{[2]}$, ...,$v^{[2m]}$ form a real basis $V_m \in \RR^{\infty \times 2 m}$ of ${\cal K}_R^{2m}(\kappa)$.  The reduced-order solution is now  expanded as $u_m=V_m \bsz$ with expansion coefficients $\alpha_i$ collected in vector $\bsz=[\alpha_1,...,\alpha_{2m}]^T$. These coefficients can be obtained from a standard Galerkin procedure defined through the weak form of   (\ref{eq:weakFormHelmholtzIntegral2}) leading to
\be\label{eq:ROMsolutionz}
 \bsz= [V_{m}^H Q(s) V_{m}]^{-1} V_{m}^H b \quad \text{or} \quad \bsz= R_m^{-1} (s)V_{m}^H b
\ee
with $b=-{\delta(x-x_\text{S})}/{\nu(x_\text{S})^2}$ and where $\bsR_m(s)$ is the $2m \times 2m$ reduced-order operator given by $\bsR_{m}(s)= V_{m}^H Q(s) V_{m}$. The reduced order model is structure preserving as show in the following proposition.

\begin{proposition} \label{prop:RKSstrucPre}
The reduced-order operator $\bsR_{m}(s)$ preserves the structure of the full-order operator $Q(s)$, that is, $\bsR_{m}(s)$ is symmetric, satisfies the Schwarz reflection principle $\overline{\bsR_{m}(s)}=\bsR_{m}(\overline{s})$ and its numerical range is contained in the numerical range of $\bsQ(s)$, that is $ {\cal W} \left\{ \bsR_{m}(s) \right\} \subseteq {\cal W} \left\{ Q(s) \right\}$.
\end{proposition}

\begin{proof}
The symmetry of $\bsR_m(s)$ follows from the symmetry of $Q(s)$ since $\bsR_m(s)=V_{m}^H Q(s) V_{m}=V_{m}^T Q(s)^T \bar{V}_{m}=\bsR_m(s)^T$, since $V_m$ is a real. The same argument shows that the Schwartz reflection principle holds as  $\bsR_m(\bar{s})=V_{m}^H Q(\bar{s}) V_{m}=V_{m}^H \bar{Q}(s) V_{m}=\bar{\bsR}_m(s)$.
Moreover, from the definition of the numerical range of the reduced-order we find that
\be
\bsx_m^H \bsR_{m}(s) \bsx_m = (V_{m} \bsx_m)^H Q(s) (V_{m} \bsx_m).
\ee
Now, every $s$ that satisfies $\bsx_m^H \bsR_{m}(s) \bsx_m = 0$ also satisfies $(V_{m} \bsx_m)^H Q(s) (V_{m} \bsx_m)=0$, such that every point in the numerical range of $\bsR_{m}(s)$ is also included in the numerical range of $Q(s)$.
\end{proof}
 
Thus, preservation of passivity~(\ref{passivity}) is  guaranteed by Proposition~(\ref{prop:RKSstrucPre}), so is the preservation of causality and stability. Therefore, existence and uniqueness of the reduced-order solution in (\ref{eq:ROMsolutionz}) is also guaranteed by Proposition~(\ref{prop:RKSstrucPre}), since it guarantees that $\bsR_{m}(s)$ is invertible for all Laplace parameters $s$ with $\Re(s) \ge 0$. Finally, we mention that the time-domain counterpart of ${u}_m(s)$ can be obtained by evaluating the inverse Laplace transform using quadrature rules \cite{Weideman2007}. 

We end this section by introducing an alternative way of representing the reduced-order solution, which will be useful in the development of phase-preconditioned RKS methods. In expansion form the reduced-order solution can be written as
\be\label{eq:ComplexConjExpansion}
u_m = \sum_{i=1}^{m} \begin{bmatrix} d_i \\ \delta_i \end{bmatrix}^T \begin{bmatrix} u(s_i)\\  u(\overline{s}_i) \end{bmatrix},
\ee
where the expansion coefficients $d_i$ and $\delta_i$ follow from the Galerkin condition. Due to the above mentioned linear independence of  the real and imaginary part of the snapshots, this representation is algebraically equivalent to $u_m=V_m \bsz$, i.e. there exists a transform from the $2m$ coefficients $\alpha_i$ to the coefficients $d_i$ and $\delta_i$ of equation~(\ref{eq:ComplexConjExpansion}).

\subsection{Interpolation Properties}
The standard theory of Galerkin interpolatory-projection model reduction of passive, self-adjoint, dynamic systems yields the following interpolation properties (e.g., see \cite{InterpolatoryModelReduction}).

\begin{proposition} 
\label{prop:SISOInterpolation}  The projected RKS solution  $u_{m}(s)$ interpolates at the shifts, i.e.,
\be
\label{eq:Interpolationprob}
u_{m}(s) = u(s) \quad\quad \forall s\in \kappa \cup \overline{\kappa}
\ee
and the SISO reduced-order transfer function ${f}_m(s)$ is a Hermite interpolant of the SISO transfer function ${f}(s)$ at the shifts, that is,
\be
{f}_{m}(s)={f}(s) \quad \text{and} \quad \frac{\rm d}{{\rm d} s}{f}_{m}(s)=\frac{\rm d}{{\rm d} s} {f}(s) \text{ with } s \in \kappa \cup \overline{\kappa}.\label{Hermite}
\ee
\end{proposition}

\begin{proof}
 Since ${\cal K}^{m}(\kappa) \subset {\cal K}_R^{2m}(\kappa)$ and ${\cal K}^{m}(\overline{\kappa}) \subset {\cal K}_R^{2m}(\kappa)$, property (\ref{eq:Interpolationprob}) follows directly from the uniqueness of the Galerkin condition for passive problems. To prove (\ref{Hermite}),  we first introduce the field error and residual as
 \[
 e_m(s)=u(s)-u_m(s) \text{ and } r_m(s)=b -Q(s) u_m(s)
 \]
 respectively. From the Galerkin condition we obtain the relation
 \be
 \label{eq:GalerkinInterpol}
 u_m^H(\overline{s})  r_m(s) = 0,
 \ee
 since $u_m(\overline{s}) \in \mathcal{K}^{2m}_{\rm R}$. 
 The error of the transfer function can now be written as
 \begin{equation}
 {f}(s)-{f}_m(s) = b^H  e_m(s)=u^H(\overline{s}) Q(s)  e_m(s),
 \end{equation}
 where we have used Schwarz's reflection principle. Since $Q(s)  e_m(s) = r_m(s)$ and the Galerkin condition of equation~(\ref{eq:GalerkinInterpol}) holds, we can write 
 \be\label{eq:PropositionProof}
 {f}(s)-{f}_m(s) =u^H(\overline{s}) r_m(s) =e_m^H(\overline{s})  r_m(s),
 \ee
 which has double zeros at $s=\kappa \cup \overline{\kappa}$, since the error and residual vanishes for these frequencies due to relation~\cref{eq:Interpolationprob}.
\end{proof}

The outlined approach is most efficient if only a few singular Hankel values of the system contribute to the solution as is the case for resonating structures with a few excited and observable modes \cite{ParametricMOR}. Then the frequency-domain response is well-described by a low-degree rational function and a rational Krylov technique will therefore quickly capture the desired wavefield response. For waves characterized by large travel times; however, this may no longer be the case, since such responses are highly oscillatory in the frequency-domain and sampling should at least take place at half the Nyquist-Shannon sampling rate. As an illustration, consider a source-receiver pair with an arrival at $T^{\text{arr}}$ such that the source wavelet convolved with $\delta(t-{T^{\text{arr}}})$ is measured. In the Laplace domain this translates to multiplication by $\exp{ -sT^{\text{arr}} }$, which means that according to the Nyquist sampling theorem the maximum frequency-domain sampling distance is $\Delta s={\pi}/{T^{\text{arr}}}$ on the imaginary axis. Clearly, the number of required frequency-domain samples increases as the travel time increases leading to prohibitory large rational Krylov subspaces. In the next section we will incorporate travel time information to obtain basis functions that are less oscillatory to lower this sampling demand. 

\section{Field parametrization for SISO problems} 
To enhance the convergence of an RKS approach for travel time dominated structures, we need to incorporate travel time information into the Krylov subspace, and thus into our basis functions. To this end, we assume that variations of the medium take place on a scale much larger than the wavelength at the considered frequencies, since this allows us to use a geometrical optics ansatz. Every basis vector belonging to the RKS is now split into an incoming and an outgoing wave and for each of these waves we factor out a strongly oscillating phase term $\exp{\pm s T_{\text{eik}}}$, where $T_{\text{eik}}=T_{\text{eik}}(x)$ is the eikonal time that solves the eikonal equation 
$| \nabla T_\text{eik}(x)|^2= \frac{1}{\nu(x)^2}$. Splitting of the fields is realized using one-way wave equations. First we introduce this splitting for one-dimensional systems in section~\ref{sec:1Dsplitting} and then generalize to higher dimensions in section~\ref{sec:XDsplitting}.

\subsection{One-dimensional field parametrization}
\label{sec:1Dsplitting}
We decompose the field into an incoming and outgoing component by writing
\be\label{eq:parametrization1D}
u(s_j)=\exp{-s_j T_\text{eik} } c_\text{out}(s_j) + \exp{s_j T_\text{eik}} c_\text{in}(s_j).
\ee
For each component an oscillating phase term has been factored out and the amplitudes are determined from  the single frequency snapshot solutions $u(s_j)$ via one-way wave equations as
\begin{subequations}
\begin{align}
c_\text{out}(s_j)&=\frac{\nu}{2 s_j}  \exp{s_j T_\text{eik}} \left( \frac{s_j}{\nu} u(s_j)  -   \frac{\partial}{\partial |x-x_\text{S} |  } u(s_j) \right),\label{eq:Cout}
\intertext{and}
c_\text{in}(s_j) &=\frac{\nu}{2 s_j} \exp{- s_j T_\text{eik}} \left( \frac{s_j}{\nu} u(s_j)  +  \frac{\partial}{\partial |x-x_\text{S} | }  u(s_j) \right). \label{eq:Cin}
\end{align}
\end{subequations}
In equation~\cref{eq:Cout}, the incoming wave component of $u(s_j)$ is filtered out leaving an outgoing component for which outgoing oscillations can be factored out. In equation~\cref{eq:Cin} the situation is reversed and the outgoing component of $u(s_j)$ is filtered out. Finally, we note that using the above one-way wave equations for decomposition is equivalent to enforcing the condition
\be
\exp{ s_j T_\text{eik}} \frac{\partial}{\partial |x-x_\text{S} | } c_\text{out}(s_j) +  \exp{ - s_j  T_\text{eik}}  \frac{\partial}{\partial |x-x_\text{S} | } c_\text{in}(s_j) =0,
\ee
%projection
and the amplitudes $c_\text{out}$ and $c_\text{in}$ are spatially much smoother than the wavefield $u$, since the highly oscillatory phase term has been factored out. 

Now to obtain a field approximation at frequency $s$, instead of projecting our operator onto single frequency solutions $u(s_j)$, we project it onto the phase-corrected basis functions $\exp{-s T_\text{eik} } c_\text{out}(s_j) $ and $\exp{s T_\text{eik} }  c_\text{in}(s_j)$. This is the central idea of our approach, which preserves the interpolation properties of the RKS. In particular, by introducing the phase-preconditioned subspace as
\begin{align}
\begin{split}
\calK^{2 m}_\text{EIK}(\kappa,s) = \text{span} \{ &\exp{-s T_\text{eik}}  c_\text{out}(s_1),\dots, \exp{-s T_\text{eik} } c_\text{out}(s_m) ,\\
&\exp{s T_\text{eik}} c_{\text{in}}(s_1),\dots, \exp{s T_\text{eik}} c_\text{in}(s_m) \},
\end{split}
\end{align}
%\begin{align}
%\begin{split}
%\calK^{2 m}_\text{EIK}(\kappa,s) = 
%\text{span} \{ & \exp{-s T_\text{eik}}  c_\text{out}(s_1),\dots, \exp{-s T_\text{eik} } c_\text{out}(s_m) ,\\
%{}&\exp{\phantom{-}s T_\text{eik}} c_\text{\makebox[0pt][l]{in}\phantom{out}}(s_1) ,\dots,\exp{\phantom{-}s %T_\text{eik}} c_\text{\makebox[0pt][l]{in}\phantom{out}}(s_m) \},
%\end{split}
%\end{align}
and its symmetry-preserving real form 
\be
{\calK}_\text{EIK;R}^{4 m}(\kappa,s) = {\rm span} \left\{ \Re {\cal K}_\text{EIK}^{2 m}(\kappa,s), \Im {\calK}_\text{EIK}^{2 m}(\kappa,s) \right\}
\ee 
we can construct reduced-order models in the usual way, but now in terms of frequency-{\it dependent} basis functions. More precisely, let
$M\le 4m$ be the dimension of  ${\calK}_\text{EIK;R}^{4 m}(\kappa,s)$  and let vectors $v^{[1]}(s)$, $v^{[2]}(s)$, ..., $v^{[M]}(s)\in \RR^\infty$ form an orthonormal basis of ${\calK}_\text{EIK;R}^{4 m}$, then the field approximation drawn from this subspace can be written as 
\be\label{eq:RKSExpansionAsymp}
u_m(s) = \sum_{i=1}^{M} \alpha_i(s) v^{[i]}(s)
\ee 
and the coefficients $\alpha_i(s)\in\CC$ can again be determined from the Galerkin condition. Note that $m$ denotes the number of snapshots used to construct ${\calK}_\text{EIK;R}^{4 m}$, while $M \leq 4m$ denotes the dimension of this subspace. The factor of 4 in the upper bound on $M$ is due to splitting into incoming and outgoing fields, that can lead to a twice as large approximation subspace compared with unpreconditioned RKS with the same shifts. However, as we shall see in subsection~\ref{SVD}, this increase can be circumvented in our implementation; the overall dimension of the preconditioned RKS is usually comparable to the dimension of a standard unpreconditioned RKS for the same accuracy, while using less snapshots.

With $V_{m;\text{EIK}}(s)\in \mathbb{R}^{\infty\times M}$ the real, orthonormal basis matrix of ${\calK}_\text{EIK;R}^{4 m}(\kappa,s)$, the reduced-order model that follows from the Galerkin condition can be written as a self-adjoint, time-invariant dynamic system
\be
\label{eq:ProjectionEik}
V_{m;\text{EIK}}(s)\bsR_{m;\text{EIK}}(s) V_{m;\text{EIK}}^H(s) u_{m;\text{EIK}}(s)=  b_m,
\ee
with
\be
b_m=V_{m;\text{EIK}}(s)V_{m;\text{EIK}}^H(s)b \quad \text{and} \quad \bsR_{m;\text{EIK}}(s) =V_{m;\text{EIK}}^H(s) Q(s) V_{m;\text{EIK}}(s). \nonumber
\ee
The reduced-order model of equation~\cref{eq:ProjectionEik} is the phase-corrected counterpart of the reduced-order model of equation~\cref{eq:ROMsolutionz}. Furthermore, since $\overline{c}_\text{out/in}(s_i)={c_\text{out/in}(\overline{s}_i)}$ holds because the Schwarz reflection principle is satisfied, we can also express the reduced-order model of equation~$\cref{eq:RKSExpansionAsymp}$ in terms of the amplitude functions $c_{\text{in}}(s)$ and $c_{\text{out}}(s)$ as (cf. equation~(\ref{eq:ComplexConjExpansion}))
\be
\label{eq:Conjugate1Dexpansion}
 u_m(s) = \sum_{i=1}^{m} \begin{bmatrix} a_i(s)\\ \alpha_i(s) \end{bmatrix}^T 
\begin{bmatrix}  \exp{-s T_\text{eik}} c_\text{out}(s_i)\\   \exp{-\overline{s} T_\text{eik}} c_\text{out}(\overline{s}_i) \end{bmatrix}+ \sum_{i=1}^{m} \begin{bmatrix} d_i(s)\\ \delta_i(s) \end{bmatrix}^T
\begin{bmatrix}  \exp{s T_\text{eik}} c_\text{in}(s_i)\\   \exp{\overline{s} T_\text{eik}}c_\text{in}(\overline{s}_i) \end{bmatrix}
\ee
with expansion coefficients $a_i$, $\alpha_i$, $d_i,$ and $\delta_i \in \CC$ and where we have assumed that $M=4m$. This formulation clearly shows that we use frequency independent amplitudes, preconditioned by frequency dependent phase functions, employing conjugation to preserve the symmetry of the wave equation. 

Finally, for field evaluations on the imaginary axis ($s\in i\mathbb{R}$) the above expansion can be written more compactly as 
\be\label{eq:Conjugate1Dexpansion_CinCout_Shorts}
 u_m(s) = \sum_{i=1}^{2m} 
 \begin{bmatrix} a_i(s)\\ \alpha_i(s)  \\ 
\end{bmatrix}^T 
\begin{bmatrix}  \exp{-s T_\text{eik}} c(s_i)\\   \exp{s T_\text{eik}} c(-s_i)  
\end{bmatrix} \quad \text{with $s \in i\mathbb{R}$}
\ee
and where $c(s_i)=c_\text{out}(s_i)$ for $i=1,2,...,m$ and $c(s_i)= \bar{c}_\text{in}(s_{i-m})$ for $i=m+1,...,2m$.

The following results show that the PPRKS retains the structure-preserving interpolatory-projection properties of standard RKS. 

\begin{lemma}
\label{lem:SISOstructure}
The system of (\ref{eq:ProjectionEik}) is structure preserving, i.e., $W\left\{ \bsR_{m;\text{EIK}}(s)\right\}\subseteq W\left\{ Q(s)\right\}$ on the range (column space) of $V_{m;\text{EIK}}(s)$. 
\end{lemma}

%old Proof
%\begin{proof}
%Let a nontrivial $\bsx_m$ be in the range of $V_{m;\text{EIK}}(s)$. Then {\color{black}$V_{m;\text{EIK}}(s)V_{m;\text{EIK}}^H(s)\bsx_m=\bsx_m$} and 
%\[\bsx_m^HV_{m;\text{EIK}}(s)\bsR_{m;\text{EIK}}(s)V_{m;\text{EIK}}^H(s)\bsx_m =\bsx_m^HV_{m;\text{EIK}}(s)V_{m;\text{EIK}}^H(s) Q(s) V_{m;\text{EIK}}(s)V_{m;\text{EIK}}^H(s)\bsx_m= \bsx_m^H Q(s) \bsx_m.\]
%\end{proof}

\begin{proof}
Let a nontrivial $\bsx_m$ be in the range of $V_{m;\text{EIK}}(s)$, that is $\bsx_m=V_{m;\text{EIK}}(s)\bsy_m$. Then
\[\bsy_m^H\bsR_{m;\text{EIK}}(s)\bsy_m =\bsy_m^H V_{m;\text{EIK}}^H(s) Q(s) V_{m;\text{EIK}}(s)\bsy_m= \bsx_m^H Q(s) \bsx_m.\]
\end{proof}
Thus, phase-preconditioned reduced order models can restrict the numerical range as $\bsx_m$ is in the range of $V_{m;\text{EIK}}(s)$; however, the spectrum is always contained in the projected operator.

\begin{proposition} 
\label{prop:SISOInterpolationP}  The SISO reduced-order transfer function retains the interpolation properties of the unpreconditioned RKS with the same shifts stated in Proposition~\ref{prop:SISOInterpolation}.\end{proposition}
\begin{proof}By construction, ${\calK}_\text{EIK;R}^{4 m}(\kappa,s)\supset {\cal K}_R^{2m}(\kappa)$ when $s \in \kappa \cup \kappa^{\ast}$. According to Lemma~\ref{lem:SISOstructure}, the ROM is passive given that the Galerkin problem has a unique solution. Therefore, the proof of Proposition~\ref{prop:SISOInterpolation} applies.
\end{proof}

One of the motivations to use this method is the expected fast convergence, when the parametrization of \cref{eq:parametrization1D} is valid. In that case only a few phase-corrected, smooth amplitude functions $c_\text{out/in}$ are required to approximate the wavefield. Furthermore, in the RKS method discussed in the previous section, the number of required shifts or frequencies is dependent on the largest arrival time; however, in the phase-preconditioned RKS (PPRKS) discussed above, the arrival times are factored out and the number of shifts is dependent on the complexity of the wave speed model $\nu(x)$ rather than the largest arrival time. We make this explicit in the following proposition. 

\begin{proposition} 
\label{prop:exactness1D} Let an  1D problem have $\ell$ homogenous layers. Then there exist $m \leq \ell+1$ non-coinciding interpolation points, such that $u_{m;\text{EIK}}(s)=u$.
\end{proposition}

\begin{proof} We start by noting that if the regions to the left and to the right of the source are considered as separate layers, then the solution to the one-dimensional wave equation consists of a superposition of left- and right-going waves of constant amplitude in each of the $\ell+1$ layers. For one-dimensional problems, the decomposition direction coincides with the travel direction of the wave; thus, $c_\text{out}(\kappa)$ and $c_\text{in}(\kappa)$ are piecewise constant. A piecewise-constant function with $\ell+1$ layers can be exactly represented by at most $\ell+1$ linear-independent piecewise-constant functions with the same jump locations.
Let us prove from the opposite, i.e., assume, that there are no  $m \leq \ell+1$ noncoinciding shifts $\kappa_i$ such that $c_\text{out}(\kappa_i)$ form a basis for all possible  $c_\text{out}(\kappa)$. Then the number of  shifts $m$ yielding linear independent solutions should be less then $\ell+1$. But by assumption there must be at least a single $c_\text{out}(\kappa)$ not from the subspace. Then one can add this solution to the subspace, i.e., the true number of linearly-independent solutions is $m+1$, which contradicts the assumption that this number is $m$. Analogously, we can prove the same statement for $c_\text{in}(\kappa_i)$.

In conclusion, we proved that there exist $m \leq \ell+1$ noncoinciding shifts $\kappa_i$ such  that $c_\text{out}(\kappa_i)$ and $c_\text{in}(\kappa_i)$  form respective bases for all possible  $c_\text{out}(\kappa)$ and $c_\text{in}(\kappa)$,  i.e., the exact  solution $u$ will be in the projection subspace. Finally, due to Lemma~\ref{lem:SISOstructure} the exact solution $u$ will be the unique solution of the Galerkin problem. 
\end{proof}

The proposition can be extended to almost all arbitrary $\ell+1$ interpolation points, as the interpolation points that lead to $\ell+1$  linear-dependent functions have measure zero. Thus, phase-preconditioning allows us to obtain the exact solution with the number of interpolation points equal to the number of homogeneous layers, i.e., this number plays the same role as the problem dimensionality in a conventional RKS approach.

\subsection{Generalization to higher dimensions}
\label{sec:XDsplitting}

In higher spatial dimensions we again split the field in incoming and outgoing wave components and use a function $g(z)$ to factor out a strongly varying phase. Specifically, for two- and three-dimensional problems we write 
\be\label{eq:parametrizationXD}
u(s_j)=g(s_j T_\text{eik} ) c_\text{out}(s_j) + g(-s_j T_\text{eik}) c_\text{in}(s_j),
\ee
and project the problem onto the real and imaginary parts of the phase-preconditioned subspace
\begin{align}
\begin{split}
\calK^{2 m}_\text{EIK}(\kappa,s) = 
\text{span} 
\{
&g(s T_\text{eik} )  c_\text{out}(s_1),\dots, g(s T_\text{eik} )  c_\text{out}(s_m),\\
 &g(-s T_\text{eik}) c_\text{in}(s_1) ,\dots,g(-s T_\text{eik}) c_\text{in}(s_m)
 \},
\end{split}
\end{align}
where $g(z)=\exp{-z}/z$ in 3D, while $g(z)=K_0(z)$ in 2D, with $K_0$ the modified Bessel function of the second kind and order zero. The singular behavior of the field at the source location is factored out leading to a weaker dependence of the amplitude functions $c_\text{out/in}$ on the source location. 

In higher spatial dimensions, the field amplitudes are again obtained via one-way wave equations, but this time along the eikonal rays leading to decomposition directions $\pm \nabla T_{\text{eik}}$.  For 2D applications with $\Kb{0}{z}$ as incoming and $\Kb{0}{-z}$ as outgoing, we obtain the amplitude functions
\begin{align}
\label{eq:splitting1}
c_\text{out}(s_j) &= \frac{s_jT}{{\rm sign}(\Im(s_j)) i \pi} 
\left[\Kb{1}{-s_jT} u(s_j)  -  \Kb{0}{-s_jT} \frac{v^2}{s_j}\nabla T \cdot \nabla u(s_j) \right]
\intertext{and}
\label{eq:splitting2}
c_\text{in}(s_j) &= \frac{s_jT}{{\rm sign}(\Im(s_j)) i \pi} 
\left[\Kb{1}{s_jT} u(s_j)  +  \Kb{0}{s_jT} \frac{v^2}{s_j}\nabla T \cdot \nabla u(s_j) \right].
\end{align}
Analogous to the one-dimensional reduced-order solution of equation~\cref{eq:Conjugate1Dexpansion}, we can write the reduced-order solution in higher spatial dimensions as
\be
\label{eq:ConjugateNDexpansion2}
 u_m = \sum_{i=1}^{m} \begin{bmatrix} a_i(s)\\ \alpha_i(s) \end{bmatrix}^T 
\begin{bmatrix}  g(s T_\text{eik}) c_\text{out}(s_i)\\   g(\overline{s} T_\text{eik}) c_\text{out}(\overline{s}_i) \end{bmatrix}+ \sum_{i=1}^{m} \begin{bmatrix} d_i(s)\\ \delta_i(s) \end{bmatrix}^T
\begin{bmatrix}  g(-s T_\text{eik}) c_\text{in}(s_i)\\   g(-\overline{s} T_\text{eik})c_\text{in}(\overline{s}_i) \end{bmatrix},
\ee
where the coefficients follow from the Galerkin condition.

Lemma~\ref{lem:SISOstructure} and  Proposition~\ref {prop:SISOInterpolationP} can be straightforwardly extended to the multidimensional case. However, Proposition~\ref{prop:exactness1D} is not directly extendible to the multidimensional case. As opposed to the one-dimensional case, a decomposition direction does not necessarily coincide with the travel direction of the wave and the field parametrization may be poor in such cases. This problem can be resolved, however, by considering MIMO wavefield systems with multiple sources and receivers, since in this case we have a decomposition direction for each source and the span of these directions may properly capture the propagation direction of the waves. In the next section, we therefore focus on wavefield systems with multiple sources and multiple receivers. The problem may be additionally complicated by multivalued solutions of the eikonal equation. In most situations it is sufficient  to use the rays corresponding to the minimal travel time; however, as we shall see for the case of internal resonant structures in section~\ref{sec:resonant},  it can be beneficial to split the subspace along multiple rays.

\section{Phase-preconditioning for MIMO systems}
\subsection{Formulation of a block method}
The time-domain equations governing MIMO systems are given by
\be
\label{eq:contWEQMIMO}
\Delta {\rm u}^{[l]} - \frac{1}{\nu^2}{\rm u}^{[l]}_{tt} = -\frac{1}{\nu^2}\delta(t)\delta(x-x_\text{S}^l) , \quad\quad {\rm u}|_{t=0}=0,\, {\rm u}_t|_{t=0}=0, 
\ee 
on $\mathbb{R}^{k} \times [0, \infty[$, where the superscript $l$ is the source index with $l=1,\dots,N_\text{src}$. The weak formulation of the corresponding $s$-domain equations is (cf. equation~\cref{eq:weakFormHelmholtzVector}) 
\be
\label{eq:HelmHoltzMIMO}
p^H Q(s) u^{[l]} =- \overline{p}(x_\text{S}^l) \frac{1}{\nu(x_\text{S}^l)^2}, 
\quad \text{for $l=1,2,...,N_{\text{src}}$}.
\ee
Assuming possibly coinciding source-receiver pairs, we can define the source/receiver array as $B_\text{s}=[b^{[1]},\dots,b^{[N_\text{src}]}]$, with individual source contributions $ b^{[l]}=-{\delta(x-x_\text{S}^l)}/{\nu^2(x_\text{S}^l)}$ as columns. Equivalently, we define the array containing the fields as $U_\text{s}(s)=[u^{[1]}(s), \dots, u^{[N_\text{src}]}(s)]$ to write the MIMO equation of (\ref{eq:HelmHoltzMIMO}) as
\be
P^H Q(s) U_\text{s}= P^HB_\text{s}.
\ee
Finally, we define the MIMO transfer function of size $N_\text{src} \times N_\text{src}$ as
\be
F(s)=B_\text{s}^H U_\text{s}(s),
\ee
which is symmetric due to reciprocity of the wavefields. To introduce the reduced-order transfer function, we define a block rational Krylov subspace
\be
{\calK}^{m N_\text{src}}_\text{B}(\kappa) = {\rm span} \left\{ U_\text{s}(s_1),U_\text{s}(s_2), \dots, U_\text{s}(s_m) \right\}
\ee
and its real counterpart containing ${\calK}^{m N_\text{src}}_\text{B}(\kappa) $ and ${\calK}^{m N_\text{src}}_\text{B}(\overline{\kappa}) $ given by
\be
{\calK}_\text{B;R}^{2 m N_\text{src}}(\kappa) = {\rm span} \left\{ \Re {\cal K}_\text{B}^{m N_\text{src}}(\kappa), \Im {\calK}_\text{B}^{m N_\text{src}}(\kappa) \right\}.
\ee
The reduced-order model for the fields can now be constructed completely analogous to the SISO case. Specifically, with $V_{m}$ a basis array that spans ${\cal K}_\text{B,R}^{2 m N_\text{src} }(\kappa)$, we have   
\be
\label{eq:ROM_RKSMIMO}
U_{\text{s};m}(s)=V_{m} \bsR_{m}(s)^{-1} V_{m}^H B_s
\quad \text{with} \quad 
\bsR_{m}(s) = V_{m}^{H} Q(s) V_m.
\ee
The reduced-order transfer function now follows as 
\be
F_m(s)=B_\text{s}^H U_{\text{s};m}(s),
\ee
and it is straightforward to show that the MIMO reduced-order transfer function $F_m(s)$ is a Hermite interpolant of the MIMO transfer function $F(s)$. The proof of this statement is completely analogous to the proof of \cref{prop:SISOInterpolation}.

To formulate the phase-corrected extensions of the block-RKS method, we note that the block-RKS field approximation $u_m^{[l]}(s)$ due to a source $l$ can be written as
%\be
%u_m^{[l]}(s)=\sum_{r=1}^{N_\text{src}} \sum_{j=1}^{2 m} a^l_{rj}(s) u^r(s_j),
%\ee
\be
u_m^{[l]}(s) = 
\sum_{r=1}^{N_\text{src}}  
\sum_{i=1}^{m} 
\begin{bmatrix} 
a^{ [l] }_i \\ 
\alpha^{ {[l]}}_i 
\end{bmatrix}^T 
\begin{bmatrix}  u^{[r]}(s_i)\\ 
u^{[r]}(\overline{s}_i) 
\end{bmatrix},
\ee
with $s_j \in \kappa$. In other words, the field approximation $u_m^{[l]}(s)$ due to source $l$ is a linear combination of single frequency solutions from \emph{all} sources. A straightforward generalization of phase-preconditioning to MIMO systems is to use a field approximation $u_m^{[l]}(s)$ that is a linear combination of phase-corrected incoming and outgoing fields from \emph{all} sources. We write the field approximation as
\be
u_m^{[l]}(s)=
\sum_{r=1}^{N_\text{src}} 
\Bigg(
\sum_{j=1}^{m} 
\begin{bmatrix} 
a^{{[l]}}_{r j} \\ 
\alpha^{{[l]}}_{r j} 
\end{bmatrix}^T
\begin{bmatrix}
g(s T_\text{eik}^{[r]}) c_\text{out}^{[r]}(s_j) \\
g(\overline{s} T_\text{eik}^{[r]})c_\text{out}^{[r]}(\overline{s_j}) \\
\end{bmatrix}+
\sum_{j=1}^{m} 
\begin{bmatrix} d^{{[l]}}_{r j} \\ \delta^{{[l]}}_{r j} \end{bmatrix}^T
\begin{bmatrix}
g(-s T_\text{eik}^{[r]}) c_\text{in}^{[r]}(s_j) \\
g(-\overline{s} T_\text{eik}^{[r]}) c_\text{in}^{[r]}(\overline{s_j}) \\
\end{bmatrix}\Bigg),
\ee
where $T_\text{eik}^{[r]}$ is the eikonal solution corresponding to the $r$th source. The coefficients $a^{[l]}_{rj},\alpha^{[l]}_{rj}$ and $d^{[l]}_{rj},\delta^{[l]}_{rj}$ are found via the block-Galerkin condition. For $N_\text{src}>1$, this approach accounts for multi-directional scattering by representing the field as a linear combination of phase-corrected functions with multiple directions $\nabla T_\text{eik}^{[r]}$. 

The idea that we followed to justify the use of block Krylov methods is that the field caused by one source contains information about the field caused by a source with a different location (and frequency). In the context of phase-preconditioning, we can apply this idea a second time to obtain a block-preconditioned algorithm. This means that instead of using the phase-correction function $g( T_\text{eik}^{[r]} )$ to only correct $c_\text{out}^{[r]}(s_j)$ for each source location $r$ separately, we cross combine all phase functions $g(s T_\text{eik}^{[r_2]} )$ with all amplitudes $c_\text{out}^{[r_1]}(s_j)$ so that the index $r$ splits into $r_1$ and $r_2$. This leads to a field approximation $u^{[l]}_m (s)$ due to the $l$th source given by
\be\label{eq:Plan22}
u_m^{[l]}(s)=\sum_{r_2=1}^{N_\text{src}}  \sum_{r_1=1}^{N_\text{src}} \Bigg(
\sum_{j=1}^{m} 
\begin{bmatrix} a^{{[l]}}_{r_1 r_2 j} \\ \alpha^{{[l]}}_{r_1 r_2 j} \end{bmatrix}^T
\begin{bmatrix}
g(s T_\text{eik}^{[r_2]}) c_\text{out}^{[r_1]}(s_j) \\
g(\overline{s} T_\text{eik}^{[r_2]})c_\text{out}^{[r_1]}(\overline{s_j}) \\
\end{bmatrix}+
\sum_{j=1}^{m} 
\begin{bmatrix} d^{{[l]}}_{r_1 r_2 j} \\ \delta^{{[l]}}_{r_1 r_2 j} \end{bmatrix}^T
\begin{bmatrix}
g(-s T_\text{eik}^{[r_2]}) c_\text{in}^{[r_1]}(s_j) \\
g(-\overline{s} T_\text{eik}^{[r_2]}) c_\text{in}^{[r_1]}(\overline{s_j}) \\
\end{bmatrix}\Bigg).
\ee
The expansion coefficients are found from the block-Galerkin condition. Basis vectors in this expression can be viewed as a tensor-product of the amplitudes $c_\text{out}^{[r_1]}$ and the phase terms $g(-\overline{s} T_\text{eik}^{[r_2]})$, while the Hadamard product is used spatially.

Similarly to the multidimensional SISO case, Lemma~\ref{lem:SISOstructure} and  Proposition~\ref {prop:SISOInterpolationP} can be straightforwardly extended for the multidimensional MIMO case. Our experiments  presented in section~\ref{experiments2D} also indicate that the number of interpolation points needed for multidimensional MIMO configurations is dependent on the complexity of the wavespeed model and not the largest travel time (as proven for the one-dimensional case in Proposition~\ref{prop:exactness1D}).

\subsection{SVD truncation of the expansion amplitudes}\label{SVD}
In the case of many sources and receivers it is possible to compress the field amplitudes $c_\text{out}^{[r_1]}(s_j)$ and $c_\text{in}^{[r_1]}(s_j)$ a posteriori using a thin singular value decomposition. Using this SVD we compress the reduced-order model and remove redundancy in the expansion of equation~\cref{eq:Plan22}. In the phase-preconditioned approach redundancy of the basis occurs in two ways. First, smooth amplitude functions corresponding to different frequencies can be close to linear dependent due to the frequency dependent basis vectors in (\ref{eq:Plan22}). Second, the amplitude functions of multiple sources can be similar which leads to redundancy once we use cross combinations of amplitudes and phase-functions as in equation~\cref{eq:Plan22}. The original block Krylov basis does not have these redundancies.

To realize SVD truncation, the amplitudes $c_\text{in}^{[r_1]}(s_j)$ and $c_\text{out}^{[r_1]}(s_j)$ are first normalized in pairs to have unit Euclidean norm. To be more specific,  $c_\text{in}^{[r_1]}(s_j)$ and $c_\text{out}^{[r_1]}(s_j)$ are normalized by $$\sqrt{ ||c_\text{in}^{[r_1]}(s_j)||^2 + ||c_\text{out}^{[r_1]}(s_j)||^2}$$ such that the sum of the square singular values is $2 m N_\text{src}$.
In this way the ratio between the incoming and outgoing amplitude is preserved. The SVD of the $2 m N_\text{src}$ incoming and outgoing amplitudes is then computed separately and truncated after $M^\text{out}_\text{SVD}$ and $M^\text{in}_\text{SVD}$ left singular vectors to obtain the compressed amplitudes $c_\text{in;SVD}^{j}$ and $c_\text{out;SVD}^{j}$. The original amplitudes $c_\text{in}^{[r_1]}(s_j)$ and $c_\text{out}^{[r_1]}(s_j)$ are associated with a specific frequency $s_j$ and source $r_1$, whereas the compressed amplitudes $c_\text{in;SVD}^{j}$ and $c_\text{out;SVD}^{j}$ are associated with a singular value. The amplitudes are therefore no longer associated with a source or frequency and the corresponding subscripts are dropped and replaced by the singular value index $j$. The resulting reduced-order solution expressed in terms of these compressed amplitudes is given by
\be \label{eq:finaldecomp}
u_m^{[l]}(s)=\sum_{r=1}^{N_\text{src}} \Bigg(
\sum_{j=1}^{M^\text{out}_{\text{SVD}} } 
\begin{bmatrix} a^{{[l]}}_{r j} \\ \alpha^{{[l]}}_{r j} \end{bmatrix}^T
\begin{bmatrix}
g(s T_\text{eik}^{[r]}) c_\text{out;SVD}^{j} \\
g(\overline{s} T_\text{eik}^{[r]}) \overline{c}_\text{out;SVD}^{j} \\
\end{bmatrix}+
\sum_{j=1}^{M^\text{in}_{\text{SVD}} } 
\begin{bmatrix} d^{{[l]}}_{r j} \\ \delta^{{[l]}}_{r j} \end{bmatrix}^T
\begin{bmatrix}
g(-s T_\text{eik}^{[r]})  c_\text{in;SVD}^{j} \\
g(-\overline{s} T_\text{eik}^{[r]})  \overline{c}_\text{in;SVD}^{j} \\
\end{bmatrix}\Bigg),
\ee
where $M^\text{out/in}_{\text{SVD}} \ll m N_\text{src}$. 

If we contract the outgoing amplitudes and the conjugate of the incoming amplitudes into one amplitude basis $c^j_\text{SVD}$, compute the SVD after pairwise normalization and evaluate on the imaginary line {\color{black}($s\in i\mathbb{R}$)} we can expand the field as
\be
\label{eq:finaldecomp_compact}
u_m^{[l]}(s)=\sum_{r=1}^{N_\text{src}}
\sum_{j=1}^{M_{\text{SVD}} } 
\begin{bmatrix} 
a^{ {[l]}}_{r j} \\ 
\alpha^{ {[l]}}_{r j}\end{bmatrix}^T
\begin{bmatrix}
\phantom{-}g(s T_\text{eik}^{[r]})  c_\text{SVD}^{j} \\
g(-s T_\text{eik}^{[r]}) \overline{c}_\text{SVD}^{j} 
\end{bmatrix} 
\quad \text{with $s \in i\mathbb{R}$}
\ee
where $M_\text{SVD} \ll m N_\text{src}$. 

Here the $c^j_\text{SVD}$ are first $M_\text{SVD}$ left singular vectors of 
\[[\bar{c}_\text{in}^{[r_1]} \,\,\, {c}^{[r_1]}_\text{out} ]/\sqrt{||{c}^{[r_1]}_\text{in}||^2 + ||{c}^{[r_1]}_\text{out}||^2 }.\]
In our numerical experiments we show that the singular values of the contracted amplitudes {\color{black}$c_\text{SVD}^{j}$} decay much faster than the ones of the block Krylov basis. In Figure~\ref{fig:SVDSingVal} we plotted the decay of singular values of a matrix with pairwise normalized vectors $c_\text{SVD}^j$ and normalized vectors $u^{[r_1]}(s_j)$ as columns. Each SVD trace is normalized to the largest singular value to emphasize the decay. The singular values associated with $c_\text{SVD}^j$ show a strong decay with a plateau at the level of the finite-difference error, whereas the singular values associated with the wave field solutions $u^{[r_1]}(s_j)$ barely show any decay before reaching {\color{black}the} Nyquist sampling rate.

The compressibility of the amplitudes confirms that the chosen parametrization of the {\color{black}wavefield} is valid. Phase-preconditioning enhances the convergence of RKS not by increasing the subspace but by preconditioning a small basis of problem specific amplitudes in dependence of the evaluation frequency using phase functions. In our experiments we show that the number of contributing amplitudes is only weakly dependent on the number of sources.

The compression of the amplitudes offers {\color{black}several} advantages. {\color{black}First,} it significantly reduces the cost of evaluating the reduced-order model, since it reduces the amount of inner products that need to be computed to obtain the reduced-order operator. {\color{black}Second,} the cost associated with communicating and storing the reduced-order model is reduced as well. The compressed amplitude basis is only very weakly dependent on the source locations used to construct it. Therefore, we can reduce the number of sources (right-hand sides) for which the basis vectors need to be computed, since their response can be approximated from their eikonal travel time and the basis vectors computed from other sources. It is in line with our effort to reduce the computation at every stage of our algorithm.

\section{Discrete formulation}
In this section we consider the discrete implementation of the introduced reduced-order modeling technique. Discretization and selection of the grid accuracy are addressed first, followed by a discussion on how we handle numerical dispersion. 

\subsection{Finite difference discretization}
Our basic approach is to solve equation~\cref{eq:weakFormHelmholtzIntegral2} (restated here as equation~\cref{eq:weakFormHelmholtzIntegral2Restated}) using phase-corrected single frequency solutions as expansion functions and to obtain a reduced-order solution for a complete spectral interval of interest via the Galerkin condition. 

We consider {\color{black} a rectangular domain} $\Omega\in \mathbb{R}^k$ with constant $\nu(x)$  on $\mathbb{R}^k\setminus \Omega$ and discretize the equation
\be\label{eq:weakFormHelmholtzIntegral2Restated}
\int_{\Omega} \nabla \overline{p}\cdot \nabla {u} {\text d} \Omega -  \int_{\Omega} \overline{p}\frac{s^2 }{\nu^2} { u}  {\text d} \Omega + \int_{\partial \Omega} \overline{p} D(s) {u}  {\text d} {\partial \Omega} = -\frac{1}{\nu(x_\text{S})^2} \overline{p}(x_\text{S}),
\ee
using finite differences to obtain a linear shifted system with a matrix nonlinearly depending on  $s$. Discretization of the first two terms in the above equation using a second-order accurate finite difference scheme with constant step sizes is straightforward. To discretize the third term, we approximate the DtN map $D(s)$ using nearly-optimal discrete perfectly-matched {\color{black}layers} (PML) according to \cite{PML2016SiamReview}.  The optimal Zolotarev rational approximants used for the PML construction make the size the of finite-difference problem (necessary for  accurate approximation of $D(s)$) in $\mathbb{R}^k\setminus \Omega$ negligible compared to the grid in $\Omega$'s interior. In 2D for instance, the resulting equations that need to be solved in the PML can be solved efficiently with a block-cyclic {\color{black}solver} \cite{CyclicReference}, or with a band solver after sorting the PML system to a bandwidth of $2k+1$ (in 2D).  This discretization leads to the matrix equation 
\be\label{eq:weakFormHelmholtzVectorDiscrete}
 \bsp^H \bsQ(s) \bsu = - \bsp^H \bsb,
\ee
with $\bsb$ the discrete approximation of the scaled delta function. The matrix $\bsQ(s)$ of order $N$ inherits all properties of the continuous operator $Q(s)$ and thus follows the Schwarz reflection principle, is symmetric in a bi-linear form ($\bsQ(\overline{s})$ is the adjoint of $\bsQ(s)$ in the Hermitian inner product) and has a nonlinear numerical range in the left-half of the complex $s$-plane. The single frequency solutions $\bsu(s_j)$ needed to build the rational Krylov subspace can be obtained using iterative solvers or Gaussian elimination. The eikonal equation $|\nabla T|^2=\frac{1}{\nu^2}$ is solved on the same grid using a fast marching method~\cite{SethianFFM}. The associated computational cost is negligible with respect to the cost of solving the Helmholtz equation.\\

\subsection{Realization on two grids}
For smooth media, the amplitudes $c_\text{out}^{[r_1]}$ and $c_\text{in}^{[r_1]}$ are smooth functions of the spatial coordinates, since the highly oscillatory part of the frequency-domain wavefield is factored out together with the source singularity. Therefore, linear combinations of these amplitudes can form reasonable approximations to amplitude distributions at higher frequencies. Consequently, the phase-corrected reduced-order models can extrapolate to frequencies outside the convex hull of the interpolation points.  Building a reduced-order model that extrapolates to higher frequencies has the advantage that the amplitudes can be determined on grids that are much coarser than {\color{black}grids} required by a direct method at these high frequencies. This significantly reduces the computational cost of solving the Helmholtz equation to obtain the amplitude functions $c_\text{out/in}$, since the main cost of the algorithm is associated with solving shifted systems.

To be specific, let $\bsQ_{\text{fine}}(s)$ and $\bsQ_{\text{coarse}}(s)$ denote the matrix operators obtained by discretizing the wave operator~$Q(s)$ on a fine and coarse grid, respectively. Writing the transfer function and field approximations obtained with these fine and coarse grid operators as $F_{m}(s)$ and $\bsU_{m}(s)$ and $F_{\text{c};m}(s)$ and $\bsU_{\text{c};m}(s)$, respectively, we have      
\be
\label{eq:diff}
F_{m}(s)-F_{\text{c};m}(s)= [\bsU_{m}(s)-\mathsf{P}_\text{interp}\bsU_{\text{c};m}(s)]^H \bsQ_{\text{fine}}(s) [\bsU_{m}(s)-\mathsf{P}_\text{interp}\bsU_{\text{c};m}(s)],
\ee
which is essentially a block version of equation~(\ref{eq:PropositionProof}). $\bsU_m$ and $\bsU_{\text{c};m}$ don't live on the same grid such that the interpolation matrix $\mathsf{P}_\text{interp}\in \RR^{N_\text{fine} \times N_\text{coarse}}$ makes the upper equation consistent. In this case, however, $\bsU_{m}(s)-\mathsf{P}_\text{interp}\bsU_{\text{c};m}(s)$ signifies the difference between the fine and coarse grid field solutions and the interpolation property as presented in  Proposition~\ref{prop:SISOInterpolation} obviously does not hold here. In other words, using a coarse grid for construction and a fine grid for projection leads to a loss of the interpolation property of the reduced-order model. On the other hand, we do increase the accuracy of the coarse transfer function as the errors introduced by interpolation and the coarse grid solution get squared at the interpolation points $s\in\kappa \cup \overline{\kappa}$.  

The main drawback of using coarser grids is that the numerical dispersion error increases and the analytic phase term $\exp{ \pm sT_\text{eik}}$ does not match the phase term of $\bsu(s)$ for large imaginary shifts. Fortunately, we can correct for this phase mismatch. To be precise, in the analytic case the phase term $\exp{-sT_\text{eik}}$ is used to cancel the {\color{black}high-frequency} dominant term $-s^2/\nu^2$ in the wave equation. To guarantee that this cancellation takes place in the discrete case and to match the discrete and analytic phase, we introduce the discrete finite-difference gradient matrix~$\bsD_{x_i}$ (see e.g. \cite{MimeticChew}) in all spatial directions $i=1,\dots,k$ and adjust the wave speed model from $\nu$ to $\nu'$, where $\nu'$ follows from the requirement    
\be
\exp{2s \bsT_\text{eik}^{[l]}} \sum_{i=1}^{k}|\bsD_{x_i}\exp{-s \bsT_\text{eik}^{[l]} }|^2=\frac{s^2}{{\nu'^{[l]}}^2}.
\ee
This is the discrete counterpart {\color{black}of} the (continuous) relation $\exp{2s T_\text{eik}} | \nabla \exp{-s T_\text{eik}}|^2=\frac{s^2}{\nu^2}$. This equation ensures that the {\color{black}high-frequency} dominant term $-s^2/\nu^2$ vanishes and the numerical dispersion error is minimized. Obviously, the discrete Laplace operator used to obtain {\color{black}equation~(\ref{eq:weakFormHelmholtzVectorDiscrete})} should be consistent with the discretization of $\bsD_{x_i}$ and is thus given by $\sum_{i=1}^{k} \bsD_{x_i}^T \bsD_{x_i}$.

{\color{black}Finally, we note that this dispersion correction is only accurate in the dominant direction of $\nabla \bsT_\text{eik}^{[l]}$ and only works in the reduced-order modeling framework.} In our block approach multiple directions are taken into account by incorporating multiple source locations. Therefore, errors occurring in the directions orthogonal to $\nabla \bsT_\text{eik}^{[l]}$ are corrected in the block projection framework by projection onto sources with different dominant directions.

\section{Two-dimensional experiments}\label{experiments2D}
In this section we illustrate the performance of the developed solution methods using three different two-dimensional numerical experiments. In our first set of experiments, we show the performance of the proposed preconditioning technique for wavefields in a smooth layered configuration. We simulate the same structure with and without grid coarsening to show the effects of both concepts. As a second example, we consider a non-smooth medium with jumps in the wave speed profile to illustrate that the effectiveness of preconditioning decreases as the high-frequency geometrical optics argument is no longer valid. However, the method {\color{black}still exhibits excellent approximation properties} even for non-smooth media. Finally, in the third experiment, a configuration with a resonant cavity present in a smooth geology is considered.

\subsection{A geophysical structure with a smooth wave speed profile}% -- phase-preconditioning}
To illustrate the effect of phase-preconditioning, we consider the smoothed geophysical structure illustrated in Figure~\ref{fig:ConfigurationMarmousi}. This model is obtained by smoothing a layered section of the acoustic Marmousi model \cite{bourgeois1991marmousi} with a Hanning window of width $h_{\rm Han}=200$~m leading to a discretized model of order $N=4\cdot 10^5$. Five coinciding source-receiver pairs are placed at the top boundary, where a perfectly reflecting boundary condition is imposed to model a water-air interface. 
A Ricker wavelet with a maximum in its spectrum at $\omega^{\text{peak}}=8~{\rm Hz}$ (13~ppw at 1\% cut-off frequency) is used as a source signature and a near optimal eight-layer PML \cite{PML2016SiamReview} is applied on the remaining outer boundaries to simulate outward wave propagation towards infinity. Finally, a fast marching method \cite{SethianFFM} is adopted to obtain the eikonal solution for this configuration. 
The true solution $\mathsf{u}^{[4]}$ corresponding to the fourth source from the left at the frequency corresponding to 14.6~ppw is visulized in Figure~\ref{fig:u14ppw}. This solution shows ``diving wave behavior'' and a caustic can be seen at a depth of about 750~m in the left half of the configuration. The real part of the outgoing amplitude {\color{black}$\mathsf{c}^{[4]}_\text{out}$} is depicted in Figure~\ref{fig:cout14ppw}, and {\color{black}is clearly spatially much smoother than} the original wavefield. Reflections of the wavefield can easily be identified in this amplitude plot. 

The overall time-domain errors of the block-RKS and preconditioned block-RKS reduced-order models without grid coarsening are shown in Figure~\ref{fig:ConvergenceSmoothLayers}, where we used a 500-point Fourier method to obtain a comparison solution. The overall time-domain {\color{black}error} is defined by the ratio of the r.m.s. error of all traces and the r.m.s. of the signals over \emph{all} traces. Preconditioning the RKS method significantly decreases the number of interpolation points needed to reach certain error level. To obtain an error of one percent in the time-domain, for example, the RKS algorithm need about 80 interpolation points, while only 10 interpolation points are required in the phase-preconditioned algorithm. This fast convergence is due to the construction of the WKB-like field approximations at high frequencies in PPRKS, which already provide an accurate approximation of the Green's function at high frequencies and in smooth structures as considered in this example.  

The real part of the transfer function of the leftmost source to the rightmost receiver is shown in Figure~\ref{fig:FieldExampleSmoothlayers} for structure-preserving and preconditioned block-RKS reduced-order models of order $m=20$. The phase-preconditioned model coincides with the comparison solution on the complete frequency interval of interest.  
The main oscillations present in this Greens function response are due to the direct arrival of the wave and its first reflection from the salt layer located at a depth of about 2500~m. Typically, the PPRKS method provides a smooth approximation to the field response showing only small errors in the amplitudes or at highly oscillatory reflections. The structure-preserving RKS method, on the other hand, overshoots after every interpolation point causing spiking behavior as can be clearly seen in Figure~\ref{fig:FieldExampleSmoothlayers}. 

\begin{figure}[!]
\centering
\subfigure[Section of the wave speed profile of the smoothed Marmousi model.]{\label{fig:ConfigurationMarmousi}\includegraphics[width=0.47\textwidth]{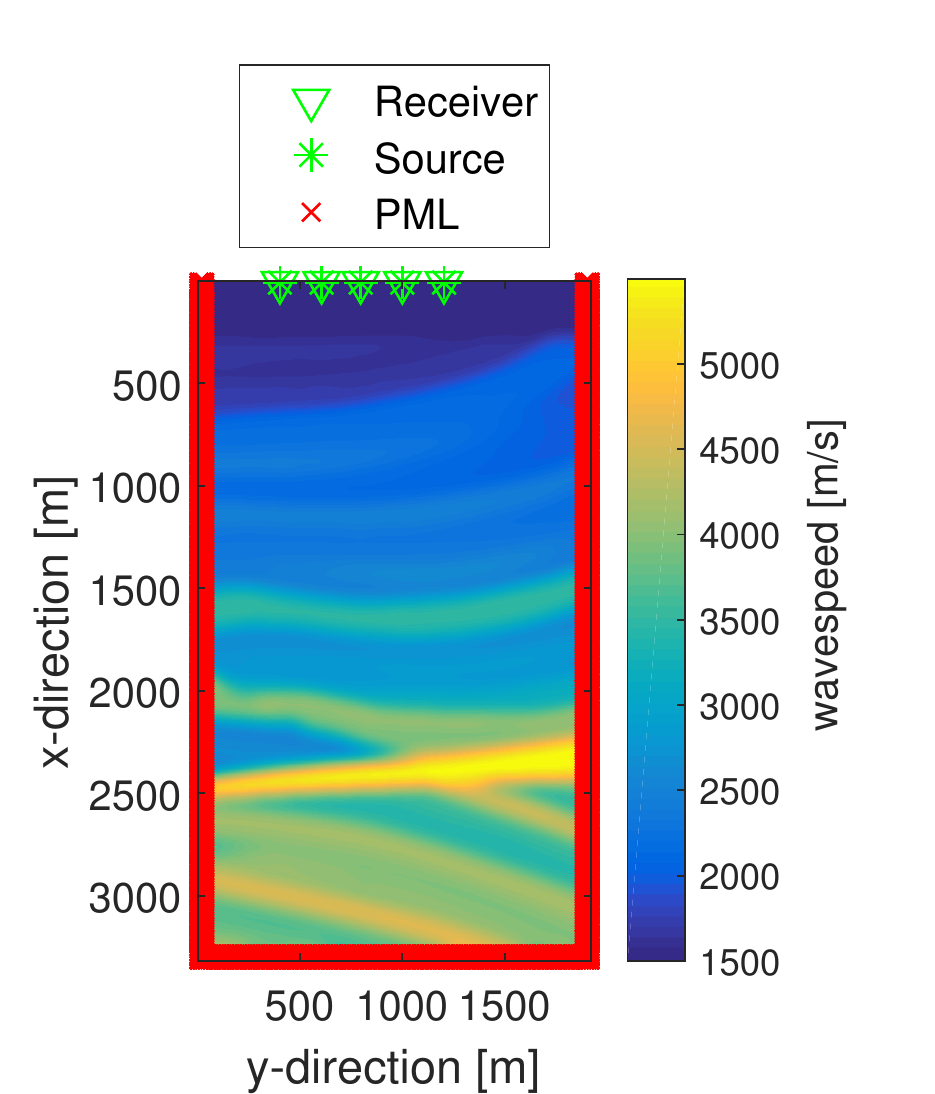}}
\subfigure[Real part of the wavefield $\mathsf{u}^{[4]}$ excited by the fourth source from the left at a frequency corresponding to 14.6~ppw.]
{\label{fig:u14ppw}\includegraphics[trim=0mm -10.5mm 0mm 0mm,clip,width=0.25\textwidth]{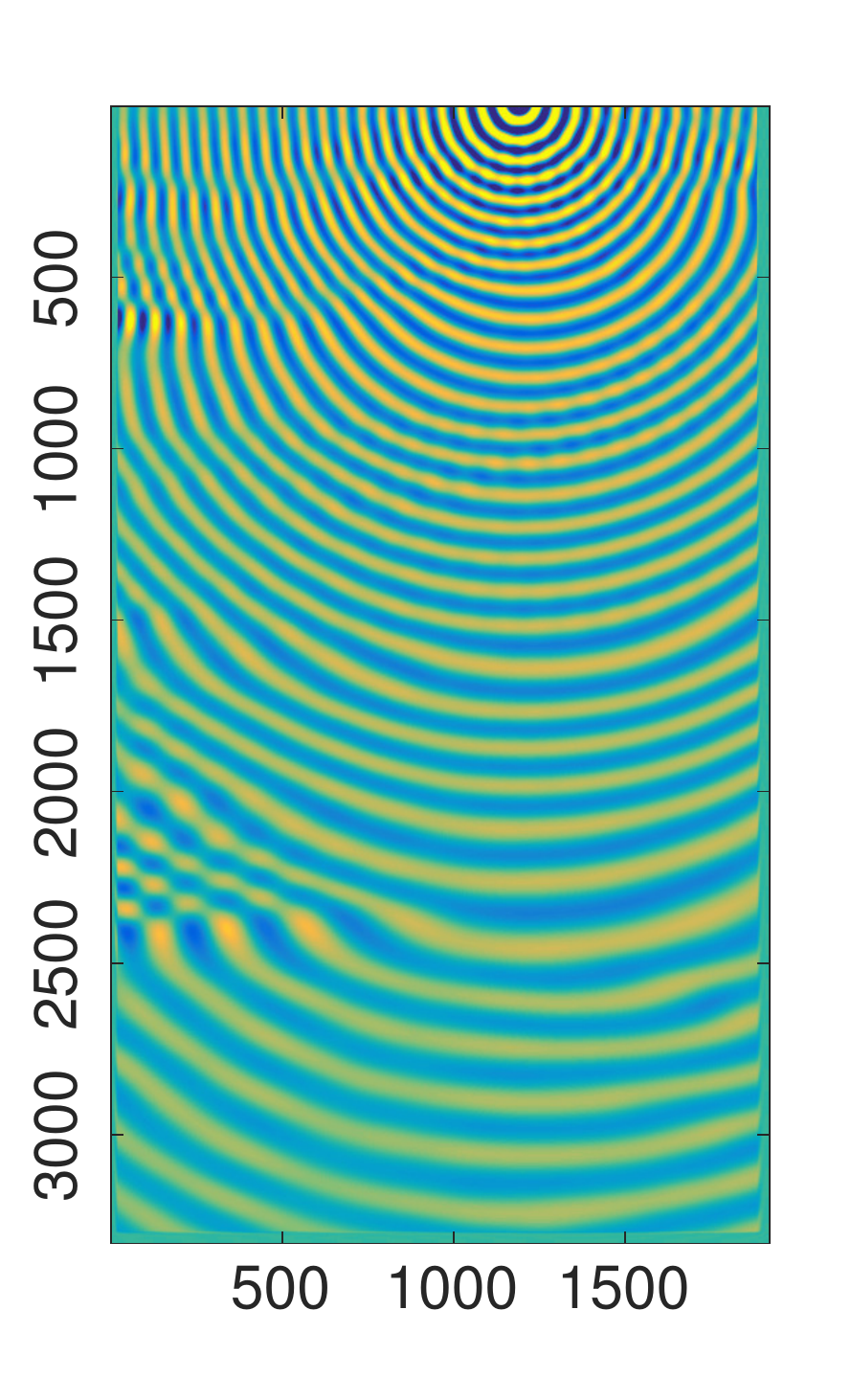}}
~
\subfigure[Real part of the amplitude $\mathsf{c}^{[4]}_\text{out}$ excited by the fourth source from the left at a frequency coresponding to 14.6~ppw.]{\label{fig:cout14ppw}\includegraphics[trim=0mm -10.5mm 0mm 0mm,clip,width=0.25\textwidth]{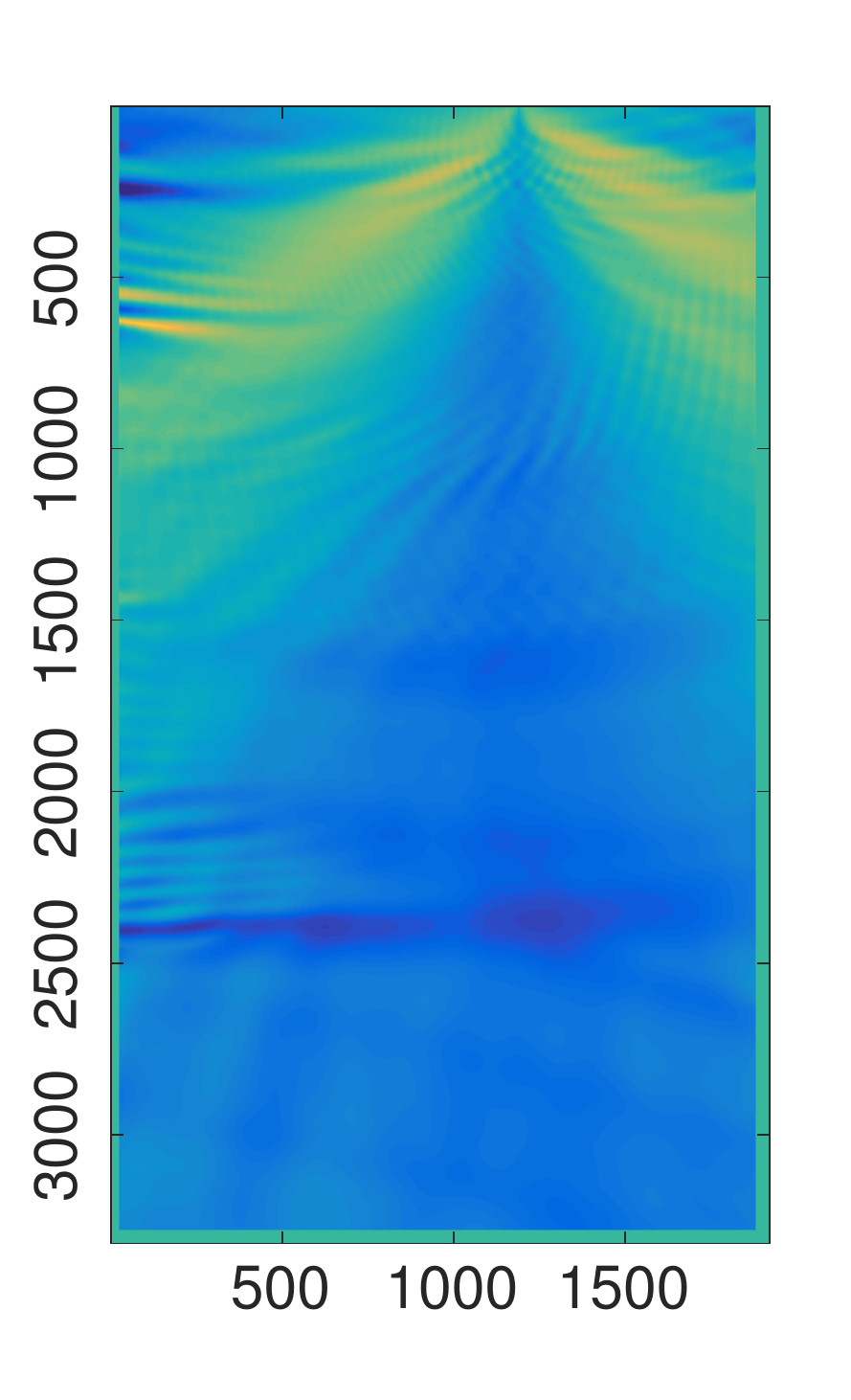}}
~
\subfigure[Time-domain convergence of {\color{black}RKS and PPRKS}. The {\color{black}overall} time-domain error is shown using a Ricker wavelet for excitation.]{\label{fig:ConvergenceSmoothLayers}\includegraphics[width=0.45\textwidth]{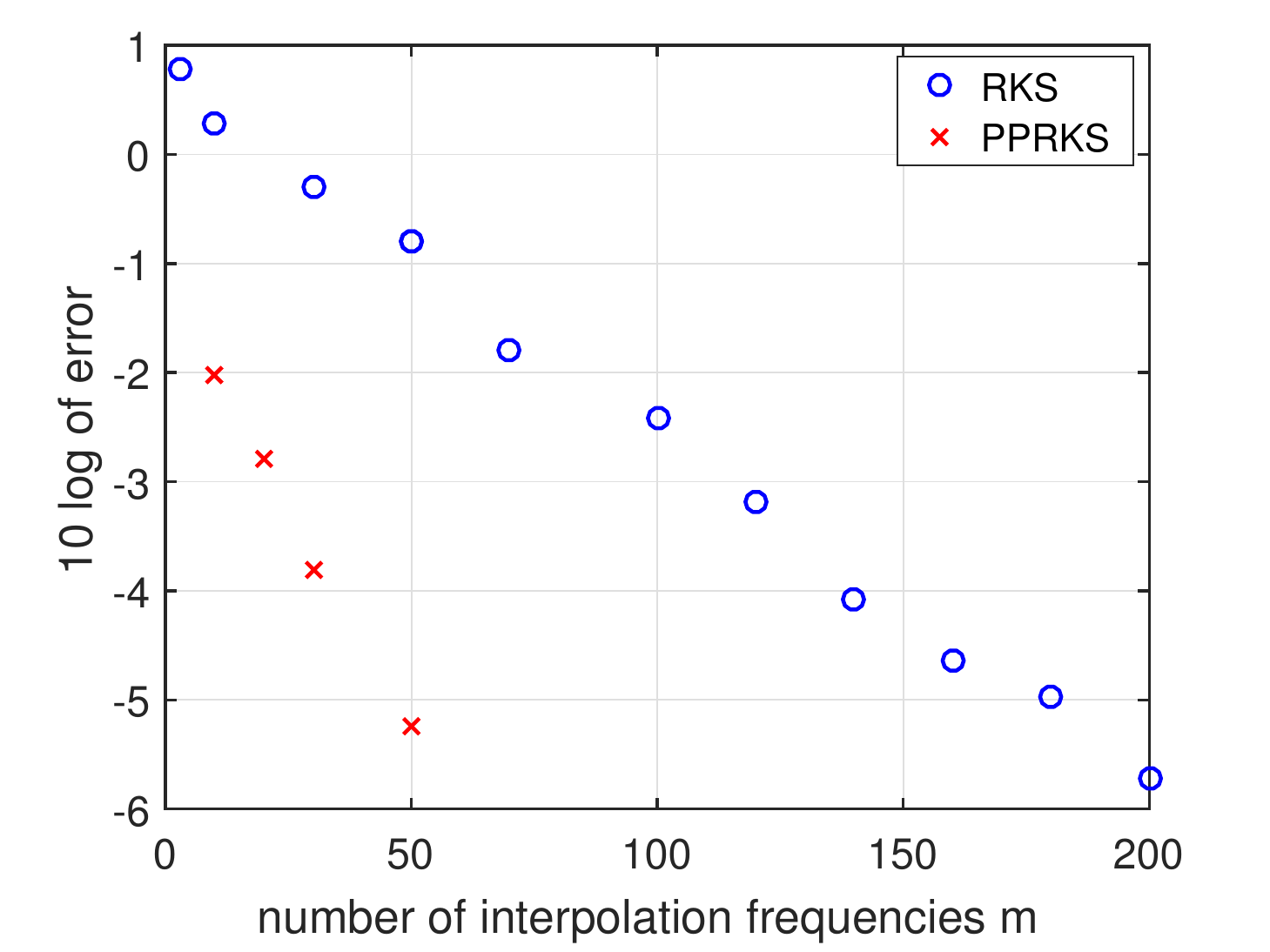}}
~
\subfigure[Real part of the frequency-domain transfer function from the leftmost source to the rightmost receiver after $m=20$~interpolation points.]{\label{fig:FieldExampleSmoothlayers}\includegraphics[width=0.45\textwidth]{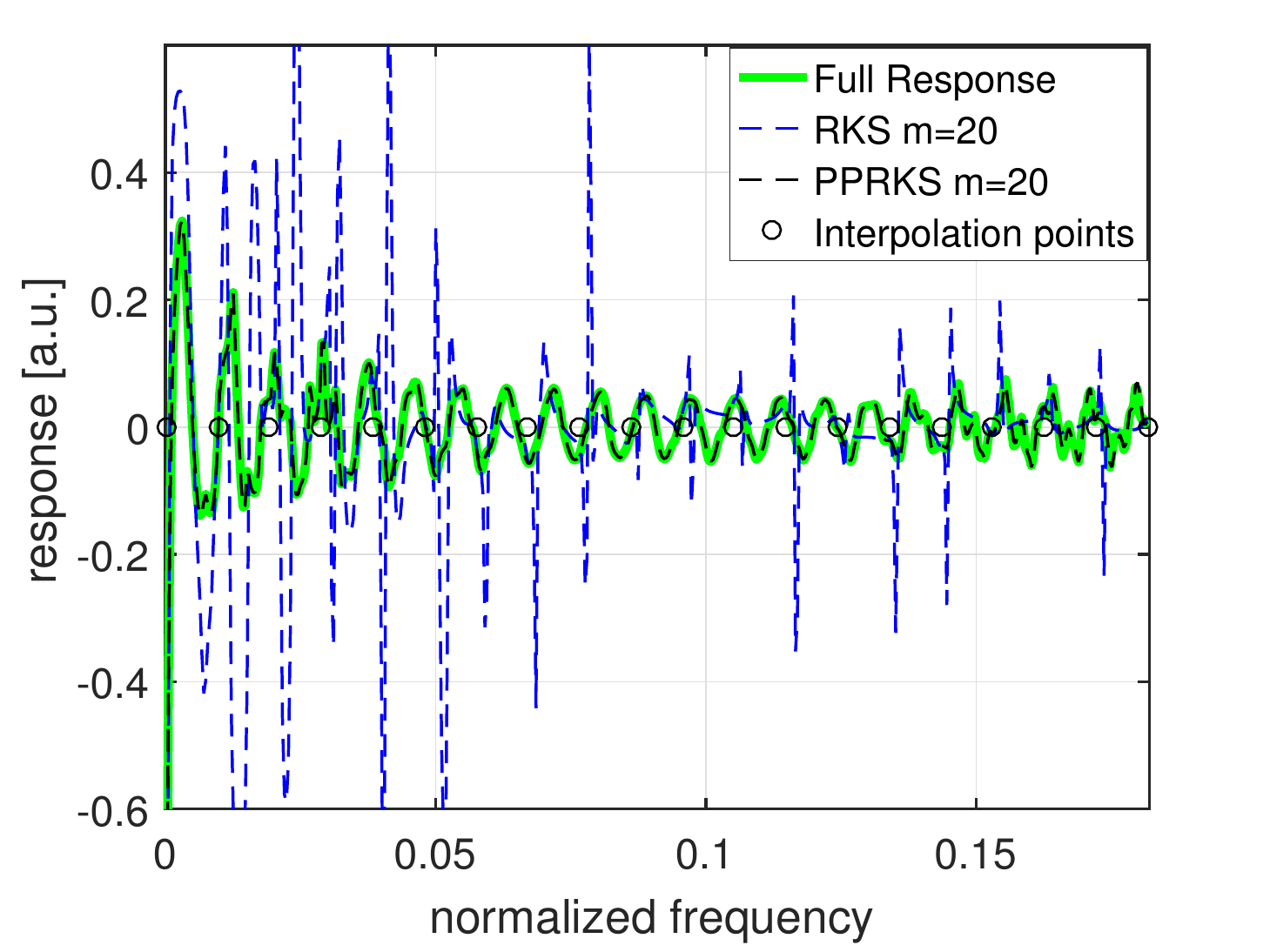}}
\caption{The smoothed Marmousi layers test configuration.}\label{fig:SmoothLayersConfigurationOverall}
\end{figure}

\subsubsection{Grid coarsening and SVD} The amplitude functions $c_{\text{in}}$ and $c_{\text{out}}$ are spatially much \linebreak smoother than the wavefield and therefore we expect that a coarser spatial grid can be employed. To investigate the effects of grid coarsening, we consider the same wave speed profile as in the previous example and place 12 coinciding source-receiver pairs at the top water-air interface instead of 5. For excitation, we use a modulated Gaussian pulse with a center frequency $\omega^{\text{peak}}$ and its support essentially given by $[0,2\omega^{\text{peak}}]$. The pulse is shifted in time such that it starts at $t=0$. Spatial discretization is now chosen such that we have about 5.5 points per smallest wavelength, where the wavelength corresponds to the center frequency of the pulse and 2.7~ppw at the cut-off frequency of the pulse. With this choice, the step sizes of the grid are four times larger than in the previous example leading to a system that is sixteen times smaller with $N=2.5 \cdot 10^4$ unknowns. Using such a coarse grid to model wavefields without phase-preconditioning is obviously insufficient, but here we expect that the smoothness of the amplitude functions $c_{\text{in}}$ and $c_{\text{out}}$ allows us to use a much coarser grid. During the evaluation of the reduced-order model we project an operator corresponding to a fine grid onto the phase-corrected RKS. For this example, we choose a fine operator using half the step size compared to the previous operator in order to show that the projection gauges the ROM to the operator used during projection.

For MIMO systems with grid coarsening, we define the error as the error averaged over all source-receiver combinations. We denote the elements of the finite difference matrix transfer function $F_{\rm F}(s)$ {\color{black}by
${f}_{\rm F}^{[ij]}(s)= b^{[i];H} {u}^{[j]}(s)$, while the element of the ROM transfer function $F_m(s)$ are given by ${f}_{m}^{[ij]}(s)= b^{[i];H} {u}_{m}^{[j]}(s)$. Having introduced these elements, the average MIMO error as function of frequency is defined as}
\be\label{eq:MOMOErrDef}
\text{err}_\text{ROM}^{\rm average}(m,s) =\frac{\sqrt{\omega_\text{max}}}{N_\text{src}^2} \sum_{j=1}^{N_\text{src}} \sum_{i=1}^{N_\text{src}}
 \frac{ \displaystyle \big|{ f}^{[ij]}_{\text{F}}(s) - { f}^{[ij]}_{m}(s)\big| \, }{\left( \displaystyle \int_{\omega=0}^{\omega_{\rm max}} \big|{ f}^{[ij]}_{\text{F}}(i \omega)\big|^2\, \text{d}\omega \right)^{1/2} }.
\ee
It is assumed that the comparison solution ${\rm f}^{[ij]}_{\text{F}}(t)$ is computed with a spatial discretization of sufficient accuracy. This averaged frequency domain error definition gives a higher error, yet delivers more insight, than computing the overall error. The overall error is dominated by the mono static elements ${f}_{\rm F}^{[ii]}$, whose direct arrival contains most energy and is well approximated. {\color{black}Furthermore, the above error definition} allows us to study the error as a function of frequency.

The phase-corrected RKS is build using $m=40$ equidistant shifts on the imaginary axis from $\omega=2.4\cdot 10^{-3}\omega^{\text{peak}}$ (2383 points per smallest wavelength) to $\omega=1.1 \omega^{\text{peak}}$ (5 points per wavelength). In other words, the RKS interpolation frequencies uniformly cover the lower half of the support of the spectrum of the source wavelet. With $m=40$ interpolation points and $N_{\text{src}}=12$ source-receiver pairs, we have 480 amplitude functions $c_{\text{in}}$ and {\color{black}an additional} 480 amplitude functions $c_{\text{out}}$. Computing the SVD of the 960 amplitude functions $[c_{\text{out}} \,\,\, \bar{c}_{\text{in}}]$, we observe that for this example, essentially only the first 100 singular functions contribute to the reduced-order {\color{black}model} for the contracted amplitudes. We therefore use a truncated SVD that uses the first 100 SVD basis functions to represent the amplitudes. 

The resulting time-domain trace from the leftmost source to the rightmost receiver is shown in Figure~\ref{fig:TD12_S1_R12_coarse} compared to the trace obtained via a 500-point Fourier method using an operator with step sizes eight times smaller than the step sizes used in the coarse operator. Both responses clearly coincide on the considered time window and the first arrival of the pulse, the complex interaction between the pulse and the upper layered medium, and the reflection of the pulse at the high contrast salt layer around $t=3000$ can be observed. The multiple reflection from source to salt layer, water/air interface, salt layer and back to the receiver can be seen around $t=6000$. 
 
In Figure~\ref{fig:SVDSingVal} we plotted the decay of the singular values of a matrix with pairwise normalized vectors {\color{black} $c_{\text{out}}^{[r_1]}(s_j)/\bar{c}_\text{in}^{[r_1]}(s_j)$} and normalized $u^{[r_1]}(s_j)$ as columns. 
The decay of the singular values of the single frequency solutions that make up the RKS is shown in black. The singular values associated with the amplitude matrix are shown in blue. Prior to the SVD the vectors $u^{[r_1]}$ were normalized, such that the square of the singular values sum up to $m N_{\text{s}}$. The vectors $c_\text{out}^{[r_1]}(s_j)$ and $\bar{c}_\text{in}^{[r_1]}(s_j)$ were normalized in pairs together to reflect the ratio of the incoming and outgoing wave at each frequency. Their SVD is computed together such that the sum of the squares of all singular values adds up to $m N_{\text{s}}$ aswell. Finally, to show the decay in singular values we normalize the largest singular value to one for each of the shown curves. 
This figure clearly shows that incoming and outgoing amplitudes are significantly compressible, whereas the RKS vectors are not. The singular values associated to RKS drop by less than a factor of 2 between the index of 50 and the index 400, indicating that the RKS can hardly be compressed. To show that after the compression the basis is {\color{black}essentially} independent of the number of sources, we computed the SVD of the amplitudes $[c_\text{out}^{[r_1]}(s_j) \,\,\, \bar{c}_\text{in}^{[r_1]}(s_j)]$ for $N_\text{src}=12,24,48 \text{ and } 96$. The number of (normalized) singular values larger than 0.01 is shown in Table~\ref{tab:SingularVal}. It shows that the number of contributing singular vectors is {\color{black}basically} independent of the number sources, so are the {\color{black}left-hand side singular vectors}.

\begin{table}[]
\centering
\caption{Number of (normalized) singular values larger than 0.01 in dependence of the number of sources}
\label{tab:SingularVal}
\begin{tabular}{|l|c|c|c|c|}
\hline
$N_\text{src}$	& 12 &24 & 48& 96\\
\hline
% Number of (normalized) singular values $>0.01$  for $c_\text{out} $ & 67 &71&71  &71  \\ 
 % Number of (normalized) singular values $>0.01$ for  $c_\text{in} $  &  95&97&82&80\\ 
    Number of (normalized) singular values $>0.01$ for  $[c_\text{out},\bar{c}_\text{in}]  $  &  69&72&73&73\\ 
  Number of (normalized) singular values $>0.01$ for $u $  & 457 &833&1369&1741\\ 
   $m\cdot N_\text{src}$ &480&960&1920&3840\\
   \hline
\end{tabular}
\end{table}

In Figure~\ref{fig:MarmousiSmoothCoarseErrTF}, the averaged error $\text{err}_\text{ROM}^{\rm average}$ over all $12^2$~traces is shown along with the interpolation points used in the construction of the reduced-order model. The same Fourier method that was used to compute the comparison solution in Figure~\ref{fig:TD12_S1_R12_coarse} is used here to compute the errors in the transfer function. Furthermore, the figure shows the error of an FDFD method with a normalized step size of 0.6, which is 20\% larger than the step size used to compute the comparison solution and used for the operator that was projected onto the PPRKS. For higher frequencies such an operator has increasing dispersion, such that the solutions between the comparison FDFD method with normalized step size of 0.5 and the one of 0.6 don't match anymore. For low frequencies there is a small discrepancy due to the inability of both grids to approximate a delta source.

When introducing grid coarsening, the ROM no longer interpolates the transfer function, but the error remains small and below 1\% on the frequency interval covered by the interpolation points. In addition, we observe that the phase-preconditioned reduced-order models can extrapolate to higher frequencies to a certain extent, since the basis in PPRKS is frequency-dependent. The error only gradually grows outside the interpolation interval, which covers the lower half of the spectrum of the pulse, and at 2.4 points per smallest wavelength we end up with an error of about 5\%.

Finally, in Figure~\ref{fig:ComparisionROMtoCoarseGrid} the averaged error in the transfer function is shown as a function of the number of points per wavelength. Again, the PPRKS with 40 interpolation points is compared to the 500-point Fourier method, but this time the latter method uses the same coarse-grid operator that is used during construction of the PPRKS (instead of the operator that uses stepsizes 0.6 as in the previous figure). Clearly, the RKS approach which uses a Galerkin condition to select optimal linear combinations with respect to a fine operator outperforms the direct Fourier method that uses the same operator to construct the field approximations. The RKS approach is gauged to the operator by using the Galerkin condition.

\begin{figure}[!]
\centering
\subfigure[Time domain trace from the left most source to the right most receiver after $m=40$~interpolation points.]{\label{fig:TD12_S1_R12_coarse}\includegraphics[trim=0mm 0mm 0mm 0mm,clip,width=0.56\textwidth]{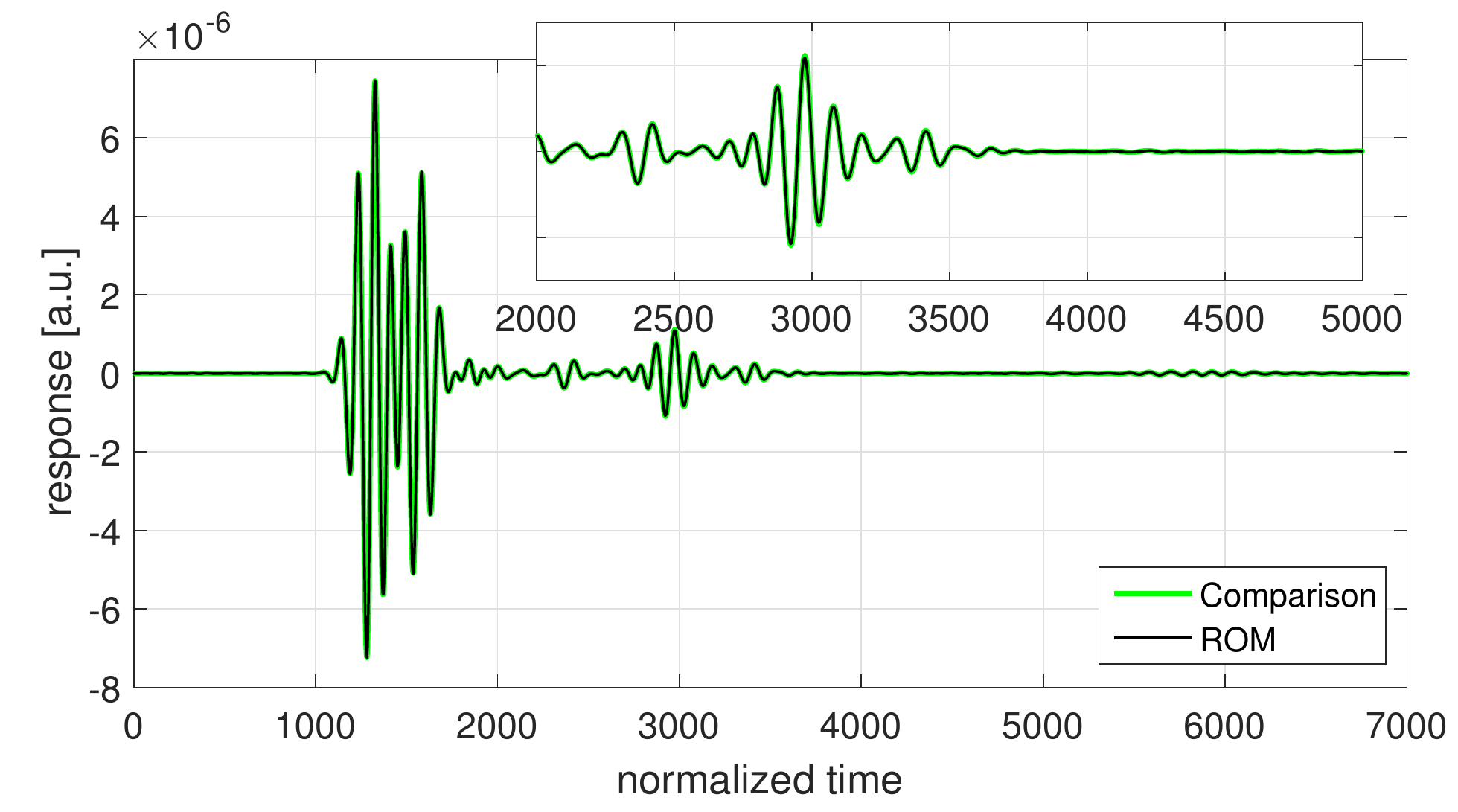}}
~
\subfigure[Decay of singular values of RKS vectors compared to the decay  of  {$[c_\text{out} \,\,\, \bar{c}_\text{in}]$}.]{\label{fig:SVDSingVal}\includegraphics[trim=0mm 0mm 0mm 0mm,clip,width=0.42\textwidth]{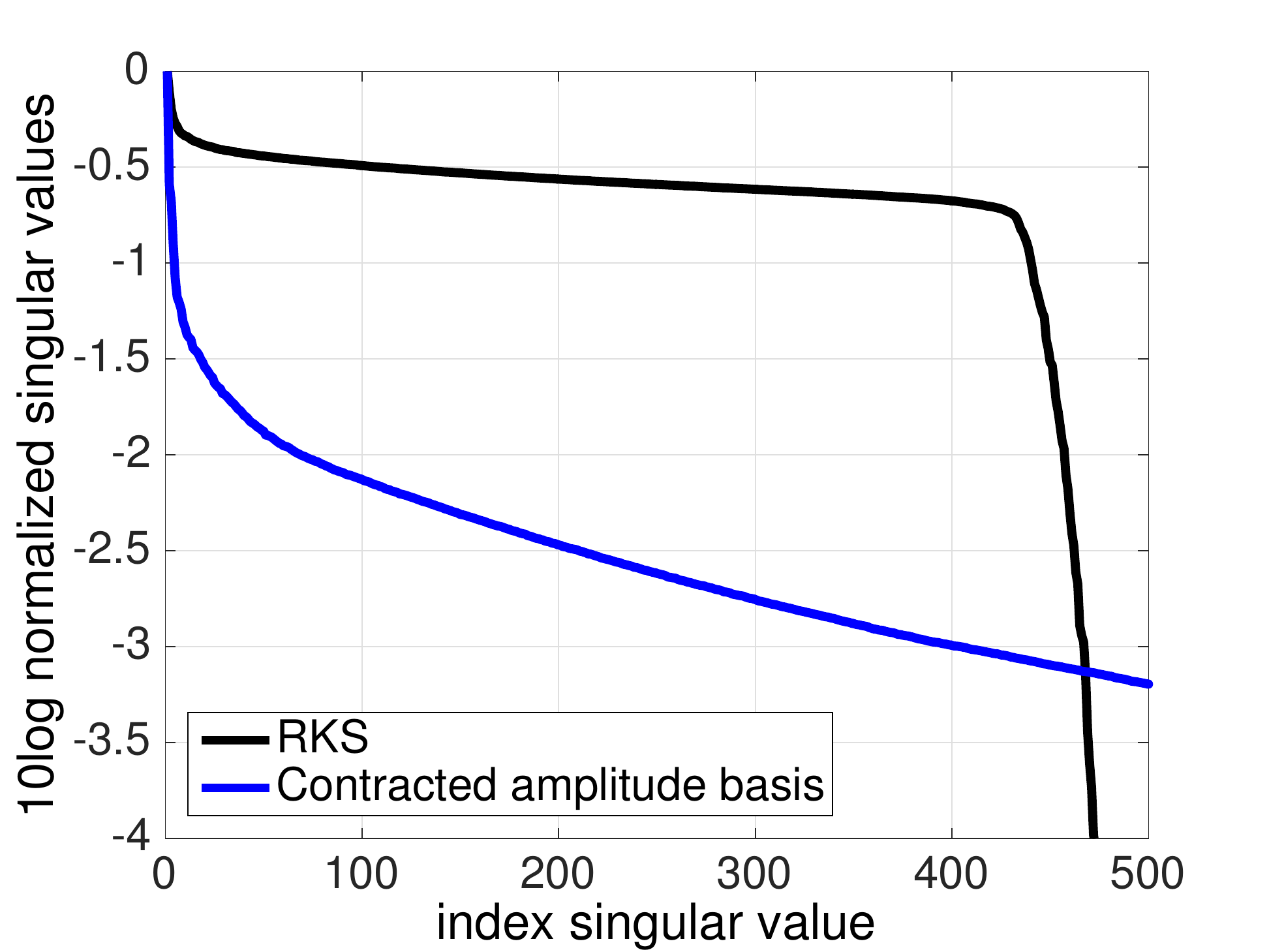}}
~
\subfigure[{\color{black}The relative error of eq.~(\ref{eq:MOMOErrDef}) of RKS, FDFD and PPRKS} as function of frequency. The error of an FDFD method with an step size that is 20\% larger then the one used to produce the comparison solution is shown in red.]{\label{fig:MarmousiSmoothCoarseErrTF}\includegraphics[trim=0mm 0mm 0mm 0mm,clip,width=0.45\textwidth]{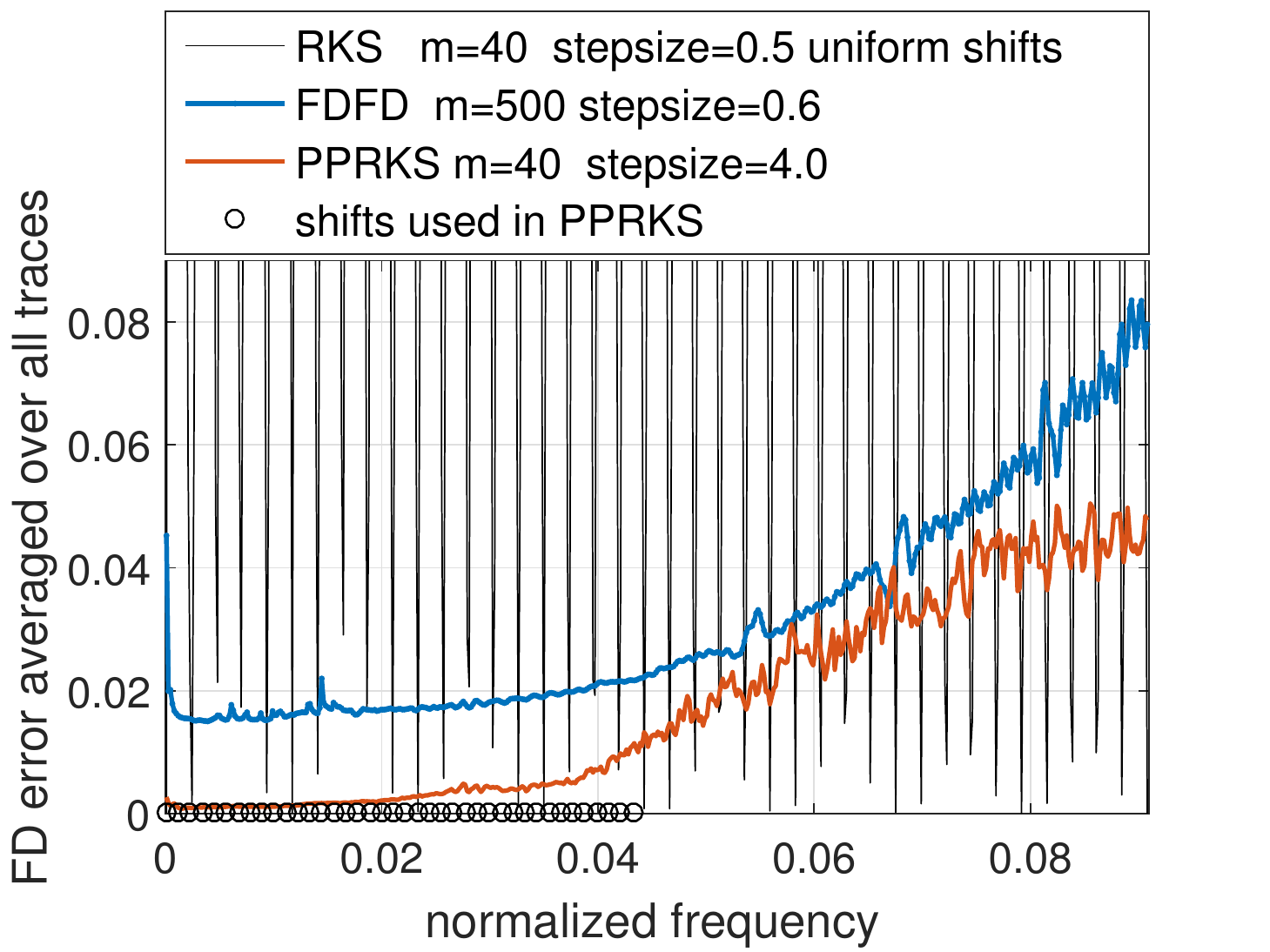}}
~
\subfigure[Comparison of the relative error of the reduced-order model $\text{err}_\text{ROM;coarse}^{\rm average}$ versus $\text{err}_\text{FD;coarse}^{\rm average}$ obtained by direct evaluation of the transfer function on the coarse grid using $\bsQ_\text{coarse}(s)$, as {\color{black}a function} of points per wavelength.  ]{\label{fig:ComparisionROMtoCoarseGrid}\includegraphics[width=0.45\textwidth]{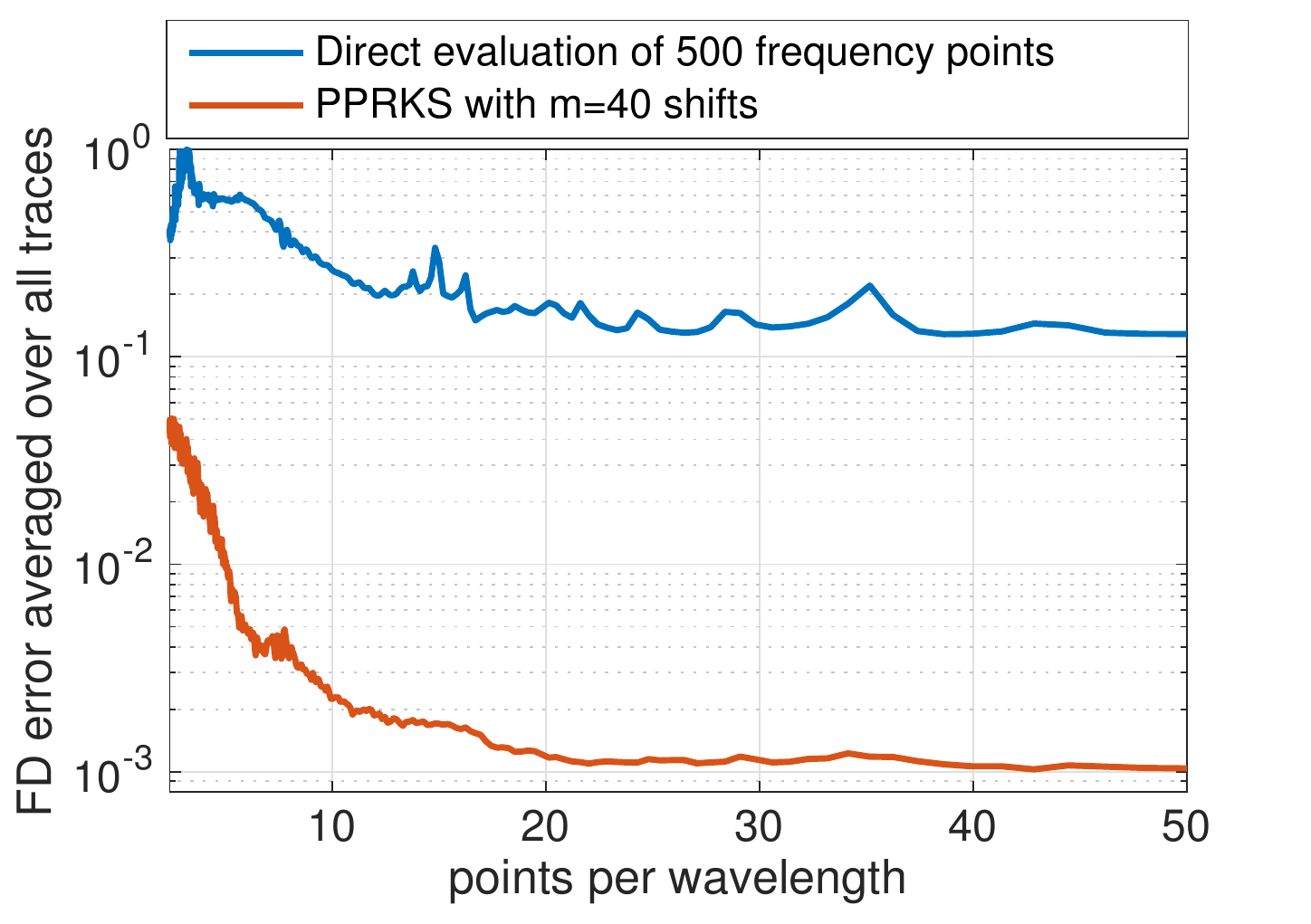}}
\caption{Smooth Marmousi test configuration with grid coarsening.}\label{fig:SmoothLayersConfigurationOverallCoarse}
\end{figure}

\subsection{A geophysical structure with a non-smooth wave speed profile}% -- phase-preconditioning}
The justification of the phase-preconditioned algorithm is based on a geometrical optics argument. This asymptotic argument is applicable for smooth media with {\color{black}spatial variations} that take place on a scale much smaller than the wavelength. On the other hand, RKS reduced-order modeling is a valid approach independent of the medium considered and \cref{prop:exactness1D} shows that one-dimensional problems with piecewise constant wave speeds need not {\color{black}be} a problem for this approach to work. Therefore, let us turn to an unsmoothed variant of the layered geophysical structure from the Marmousi model considered earlier as depicted in Figure~\ref{fig:ConfigurationMarmousiCoarseHard}.

For this structure, we essentially follow the same procedure as before. Specifically, we again position 12 source-receiver pairs at the top air-water interface and use the same coarse grid operator as in the previous example to construct a phase-corrected RKS reduced-order model of order $m=40$ with interpolation points on the imaginary axis covering the lower half of the spectrum of the pulse such that {\color{black}we have} 5 points per smallest wavelength for the highest interpolation frequency. The center frequency of the pulse is again chosen at 5.5~ppw. The only difference in model construction compared with the previous example, is that here we use a truncated SVD that takes 150 SVD basis functions into account, instead of the 100 basis functions in the previous example. Here, more basis functions are required, since the amplitude functions are less smooth due to the non-smooth wave speed profile of the present Marmousi model. Finally, the comparison solution is computed using a direct 500-point Fourier method using a spatially discretized operator with step sizes that are four times smaller than the step sizes used in the coarse operator. The coarse operator has a normalized step size of 4 and the operator used to compute the comparison FDFD response has a normalized stepsize of 1.

The resulting error is shown in Figure~\ref{fig:ComparisonFDerrCoarseSmoothFreq} along with the corresponding error curve for an FDFD method which used the coarse operator that constructed the PPRKS. {\color{black}In addition,} the error of an ordinary RKS method is shown, which uses the fine operator for construction and evaluation. It uses $m=40$ shifts uniformly distributed on the whole spectral interval. An RKS method on the fine grid interpolates the FDFD response on the shifts, which leads to a strongly oscillatory error curve. This error curve clearly shows the advantage of phase-preconditioning with a dual grid approach -- a lower error is achieved while solving considerably smaller shifted systems, and projecting on the same operator.
The performance of the algorithm for smooth profiles is better than for non-smooth wave speed profiles, especially for extrapolated frequencies. We also observe that the error decreases for lower frequencies, since lower frequencies have larger wavelengths and variations in the wave speed profile take place on a scale smaller than these wavelengths of operation. {\color{black} Furthermore,} compared to an FDFD method that uses a 20\% coarser grid than the comparison solution, the PPRKS  achieves lower errors across the whole spectral interval while the systems that need to be solved are much smaller. Especially in the area {\color{black}where} the phase-preconditioned method has shifts it reproduces the comparison solution remarkably well. A similar error comparison is shown in Figure~\ref{fig:ComparisonFDerrCoarseSmooth}, where the error is plotted against points per wavelength.

To illustrate the effects of an increased error in the time-domain, we show the time trace for the most distant source-receiver pair in Figure~\ref{fig:TD12_S1_R12_coarsehard} (for the same Gaussian pulse as used before in Figure~\ref{fig:TD12_S1_R12_coarse}) along with a comparison solution obtained with the 500-point Fourier method. We observe that the arrival times are approximated well; only the amplitudes are slightly off. Throughout our numerical work, we have found that this result is typical for non-smooth problems. Furthermore, compared with the same trace computed for smooth media as shown in Figure~\ref{fig:TD12_S1_R12_coarse}, it is clear that a larger part of the pulse is scattered back to the receiver, as arrivals are visible on the complete time interval of observation for the non-smooth velocity profile considered here. 

\begin{figure}[!]
\centering
\subfigure[Section of the wave speed profile from the Marmousi model.]{\label{fig:ConfigurationMarmousiCoarseHard}\includegraphics [trim=15mm 0mm 15mm 0mm, clip,width=0.38\textwidth]{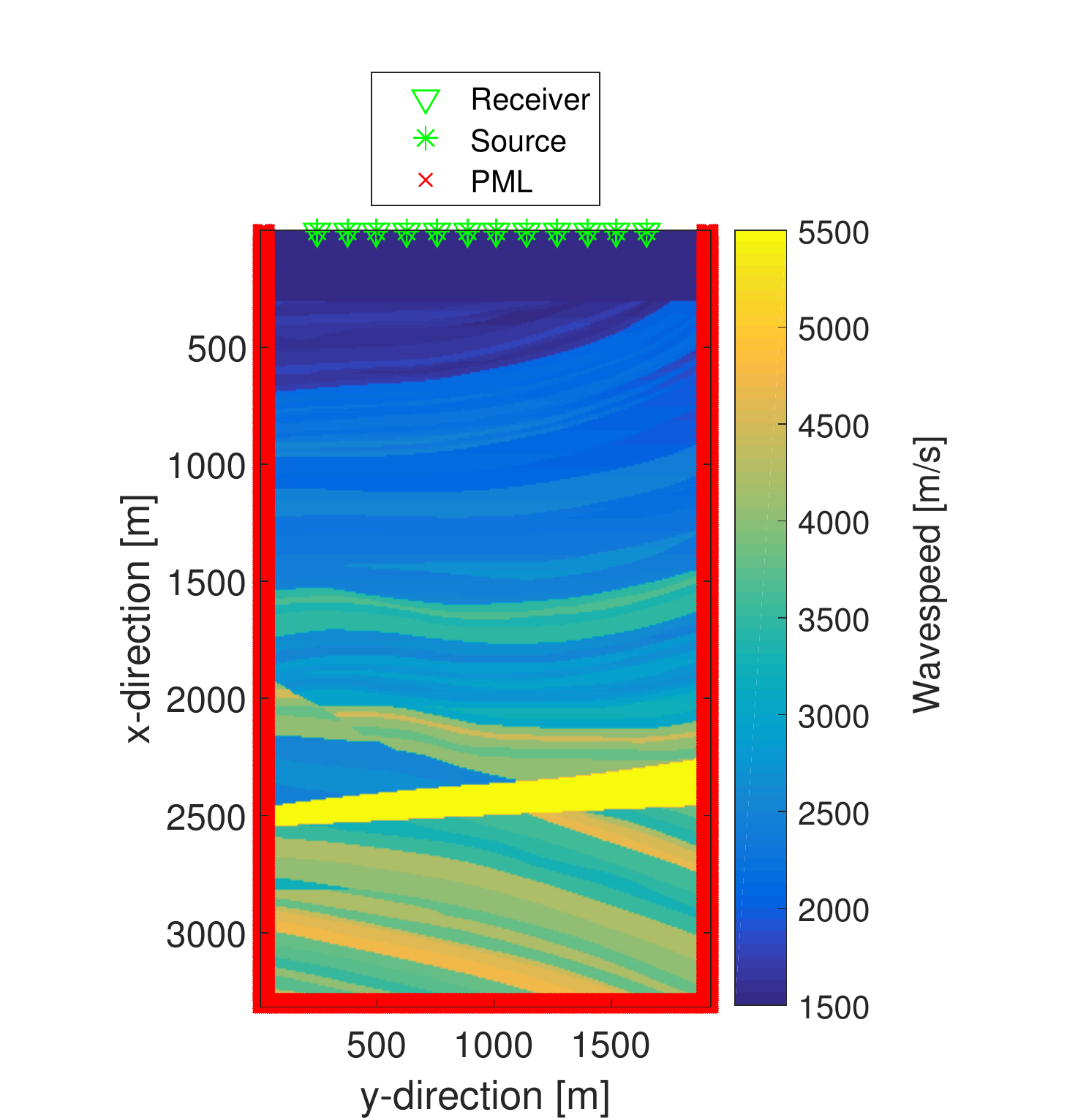}}
~
\subfigure[Comparison of relative error $\text{err}_\text{ROM;coarse}^{\rm average}$ and $\text{err}_\text{FD;coarse}^{\rm average}$  for non-smooth Marmousi layer configuration.]{\label{fig:ComparisonFDerrCoarseSmoothFreq}\includegraphics[trim=0mm 0mm 0mm 0mm,clip,width=0.59\textwidth]{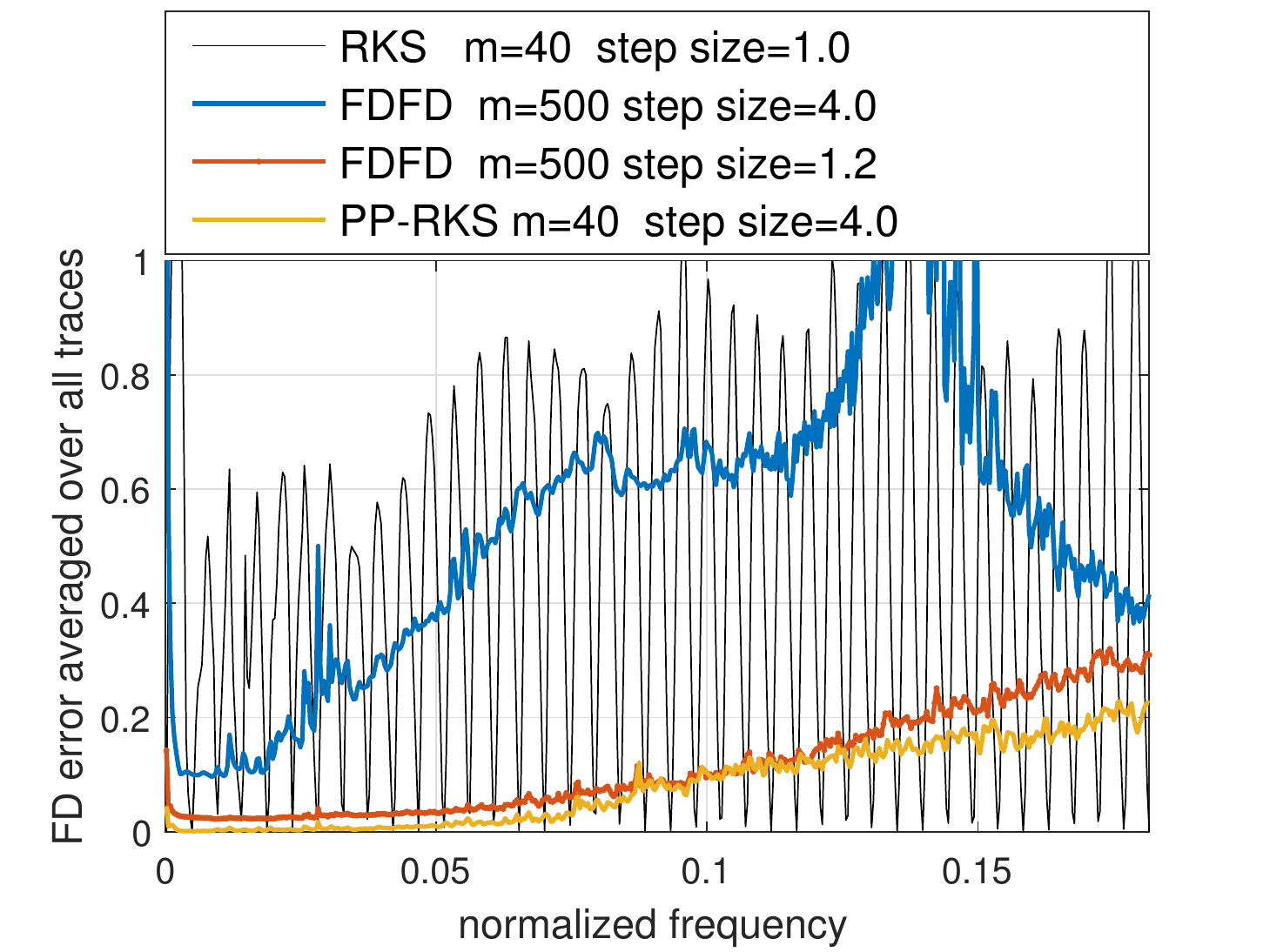}}
~
\subfigure[Comparison of relative error $\text{err}_\text{ROM;coarse}^{\rm average}$ and $\text{err}_\text{FD;coarse}^{\rm average}$ for non-smooth Marmousi layer configuration.]{\label{fig:ComparisonFDerrCoarseSmooth}\includegraphics[trim=0mm 0mm 0mm 0mm,clip,width=0.45\textwidth]{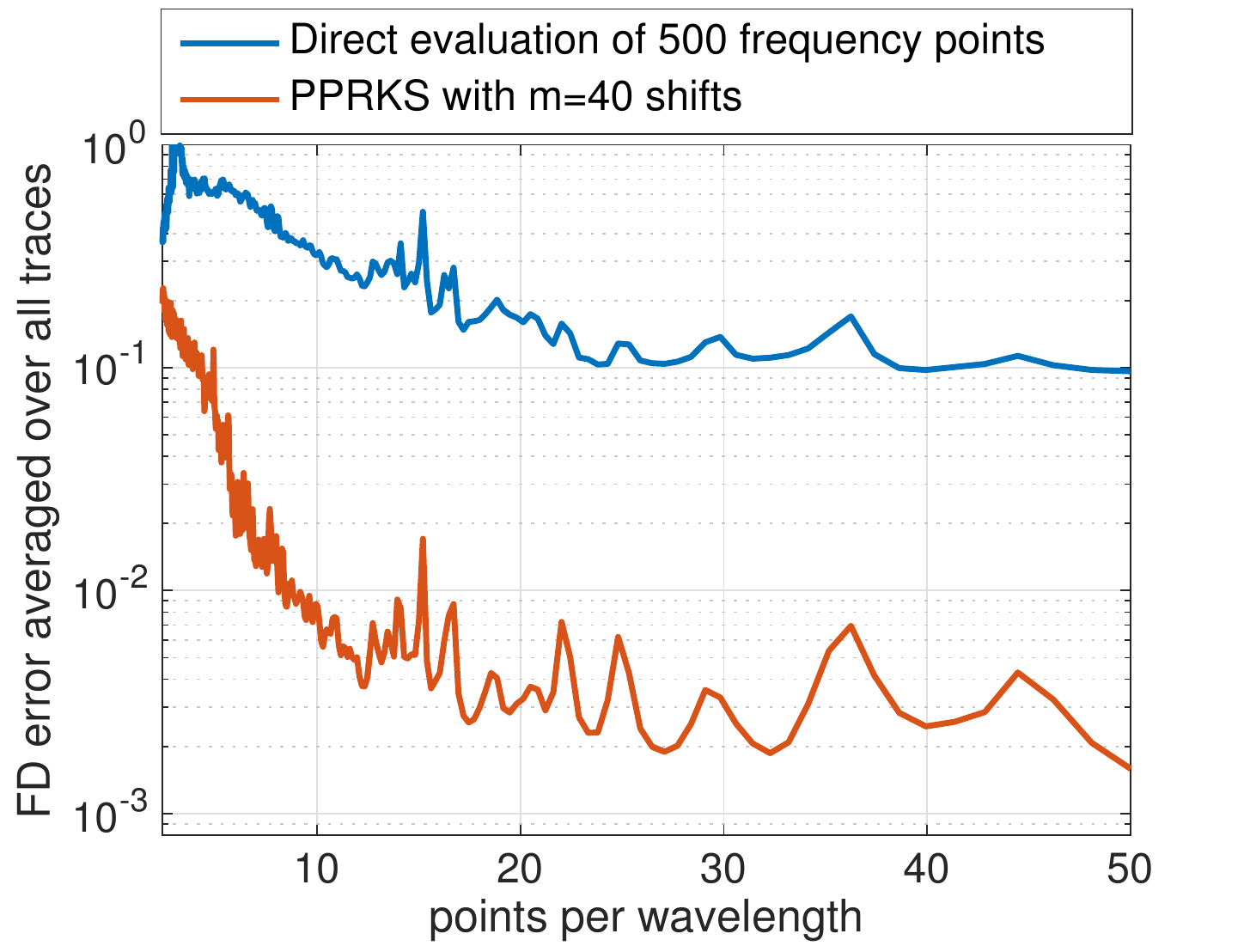}}
~
\subfigure[Time-domain trace from the leftmost source to the rightmost receiver after $m=40$~interpolation points and the comparison solution.]{\label{fig:TD12_S1_R12_coarsehard}\includegraphics [trim=0mm 0mm 0mm 4mm, clip,width=0.9\textwidth]{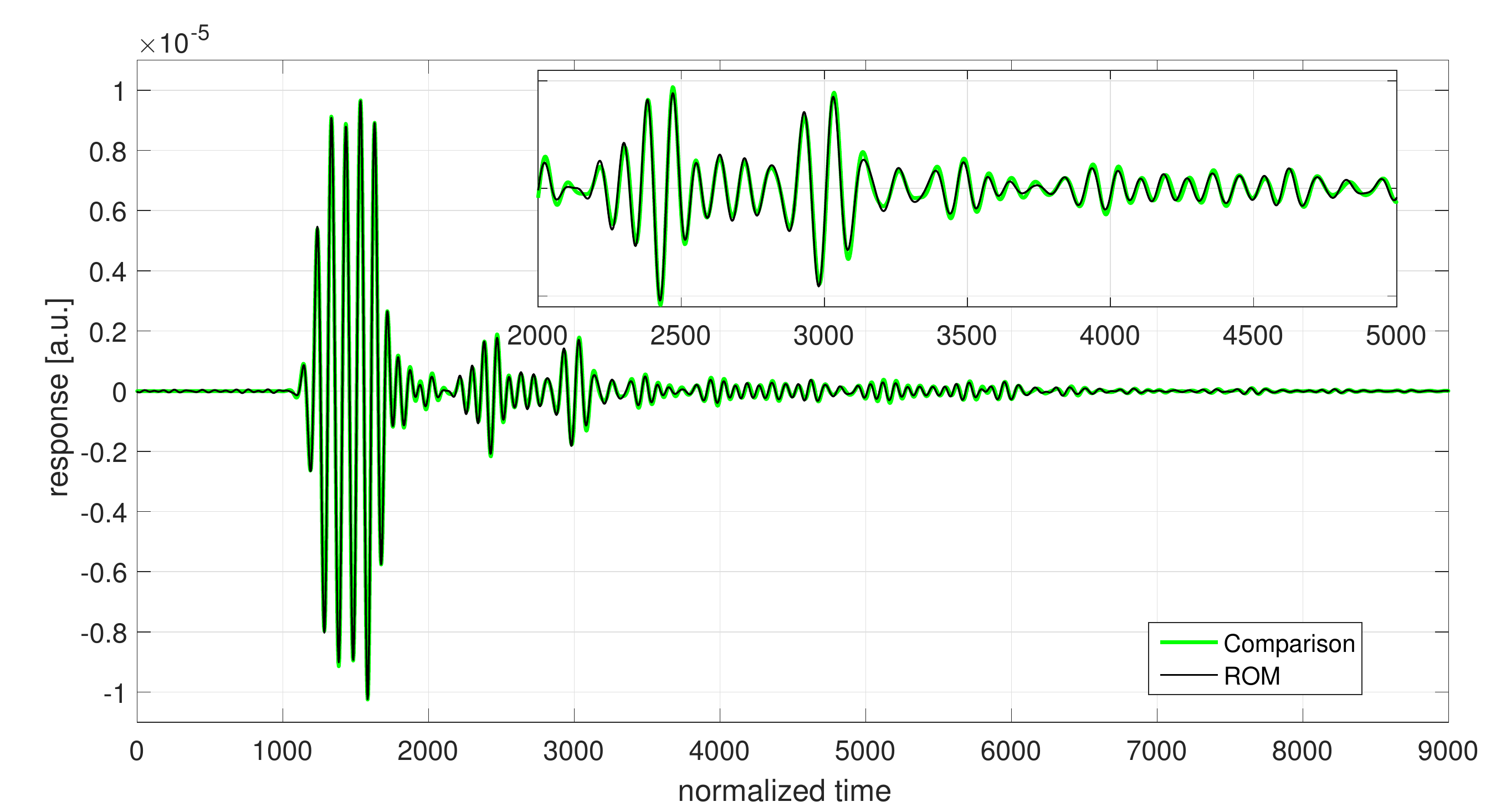}}
\caption{{\color{black}Non-smooth} Marmousi test configuration with grid coarsening.}\label{fig:
HardLayersConfigurationOverallCoarse}
\end{figure}

\subsection{A resonant cavity embedded in a smooth geology}\label{sec:resonant}
In this section we investigate the performance of our algorithm in a configuration with a resonant cavity. Figure~\ref{fig:ConfigBorehole} shows the wave speed profile, which is inspired by borehole exploration. Coinciding source-receiver pairs are placed at the surface and inside a borehole of slow acoustical wave speed. 

The grid coarsening procedure and wavelet selection is equivalent to the previous examples. A coarse grid operator with the same accuracy as selected in the previous example is used to construct a phase-corrected RKS reduced-order model of order $m=40$ with interpolation points on the imaginary axis covering the lower half of the spectrum of the pulse such that 5 points per smallest wavelength are used for the highest interpolation frequency. The center frequency of the pulse is again chosen at 5.5~ppw. 

To approximate the cavity-resonances with few interpolation points, we extend the approach discussed in this paper {\color{black}and factor out} oscillations of resonance modes along the borehole. { To do so, we take the fact that the eikonal time $T_\text{eik}$ is multivalued into account.} More specifically, each solution $\bsu(\kappa_i)$ is split using two different phase terms, a cavity-mode phase term and a propagation phase term. The eikonal phase term shown in Figure~\ref{fig:TeikBorehole1} shows a caustic inside the borehole, which has a low wave speed compared to its surrounding. In this experiment we also factor out the cavity-mode phase term $g(s T_\text{eik;CM} )$, where $T_\text{eik;CM}$ follows the borehole as shown in Figure~\ref{fig:TeikBorehole2}. The eikonal time of the cavity-mode $T_\text{eik;CM}$ is not the second arrival, but it is chosen to correctly factor out resonances present in the borehole. At every interpolation point we split the field into four amplitudes as
\begin{align}
u^{[l]}(s_j)&=g(s_j T^{[l]}_\text{eik} ) c_\text{out;eik}^{[l]}(s_j) + g(-s_j T^{[l]}_\text{eik}) c_\text{in;eik}^{[l]}(s_j),\\
u^{[l]}(s_j)&=g(s_j T^{[l]}_\text{eik;CM} ) c_\text{out;CM}^{[l]}(s_j)+ g(-s_j T^{[l]}_\text{eik;CM}) c_\text{in;CM}^{[l]}(s_j).
\end{align}
With 14 sources and 40 interpolation points we end up with 560 amplitudes for each of the four amplitudes $c_\text{in/out;{eik}/CM}$, which we compress to 30 each using an SVD. These compressed amplitudes are then used to construct the phase-preconditioned rational Krylov subspace on which the fine operator is projected. 

For these types of configurations the time window of interest tends to be very long due to the resonant nature of the configuration. FDTD therefore requires very long runtimes, {\color{black}whereas} the proposed algorithm just needs to evaluate the ROM on more frequencies to avoid aliasing. The time-domain trace of the top most source-receiver pair within the borehole is shown in Figure~\ref{fig:TD12_S1_R1_borhole}, where the emitted pulse bounces back and forth within the cavity. The reduced-order model {\color{black}captures this resonant behavior,} showing that the resonant modes are well approximated. In Figure~\ref{fig:TD12_S7_R14_borhole} a trace is shown from a source within the borehole to a surface receiver. In this trace it can be seen that the interaction of the pulse with the smooth geology is modeled {\color{black}correctly} next to the repetitive trace shape caused by the resonant cavity.
An ordinary RKS method with no grid coarsening would perform well on this problem, since it is mainly dominated by the resonant cavity; however, this would require solutions of the wave equation on a much finer grids then the proposed approach. Furthermore, contrary to the proposed approach, the RKS approximation deteriorates as the configuration size and thus the propagation distance from the cavity to the receiver increases.

In this experiment we show that the developed algorithm shows potential for reduced-order modeling of resonant cavities within slowly varying media. The combination of an RKS method together with phase-preconditioning can approximate both resonant {\color{black}eigenmodes as well} as propagative modes. We point out that this is just a first approach in order to include resonant structures into reduced-order models that are travel time dominated. 

\begin{figure}[!]
\centering
\subfigure[Simulated Configurations with labelled sources.]{\label{fig:ConfigBorehole}\includegraphics[trim=0mm 0mm 0mm 0mm, clip,width=0.31\textwidth]{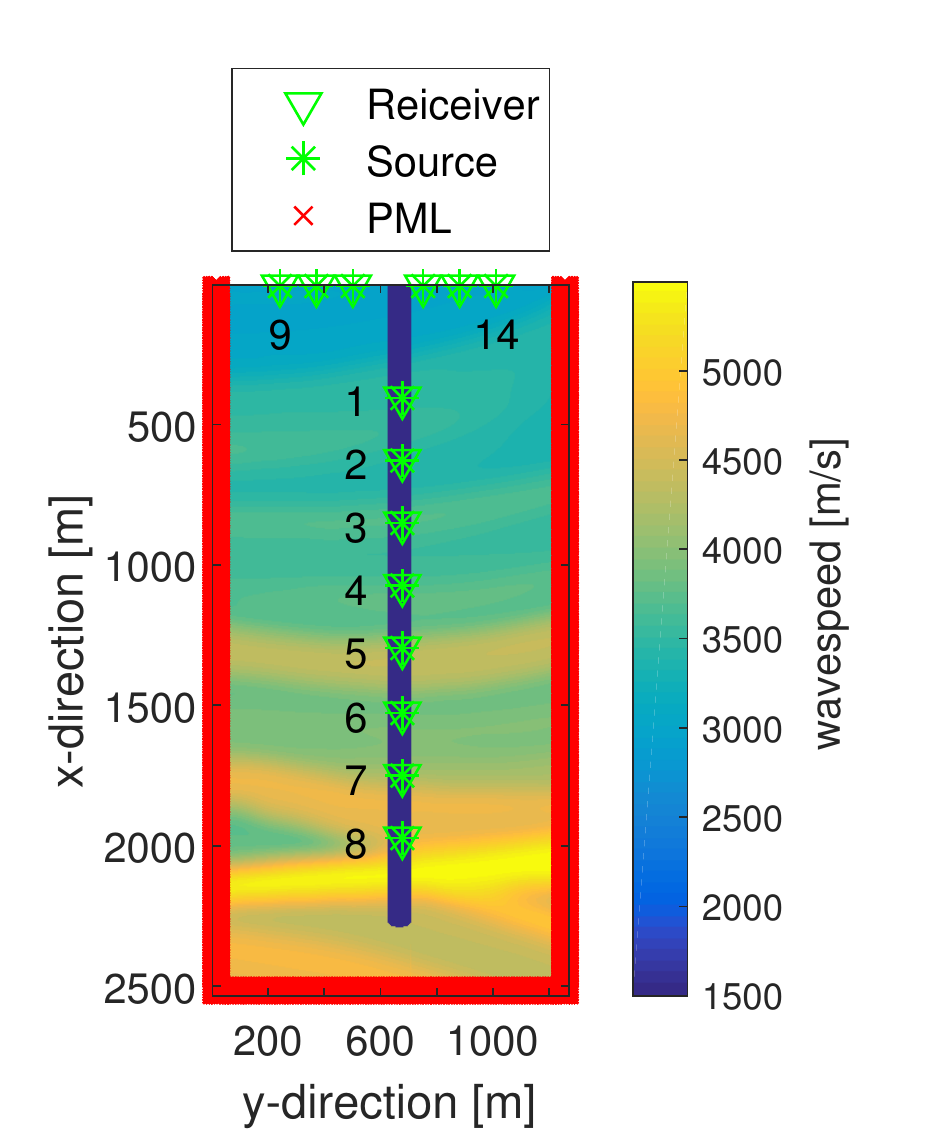}}
~
\subfigure[Isosurfaces of the eikonal time $\bsT_{\rm eik}$ used to approximate the propagative part of the solution.]{\label{fig:TeikBorehole1}\includegraphics[trim=0mm 0mm 0mm 0mm, clip,width=0.31\textwidth]{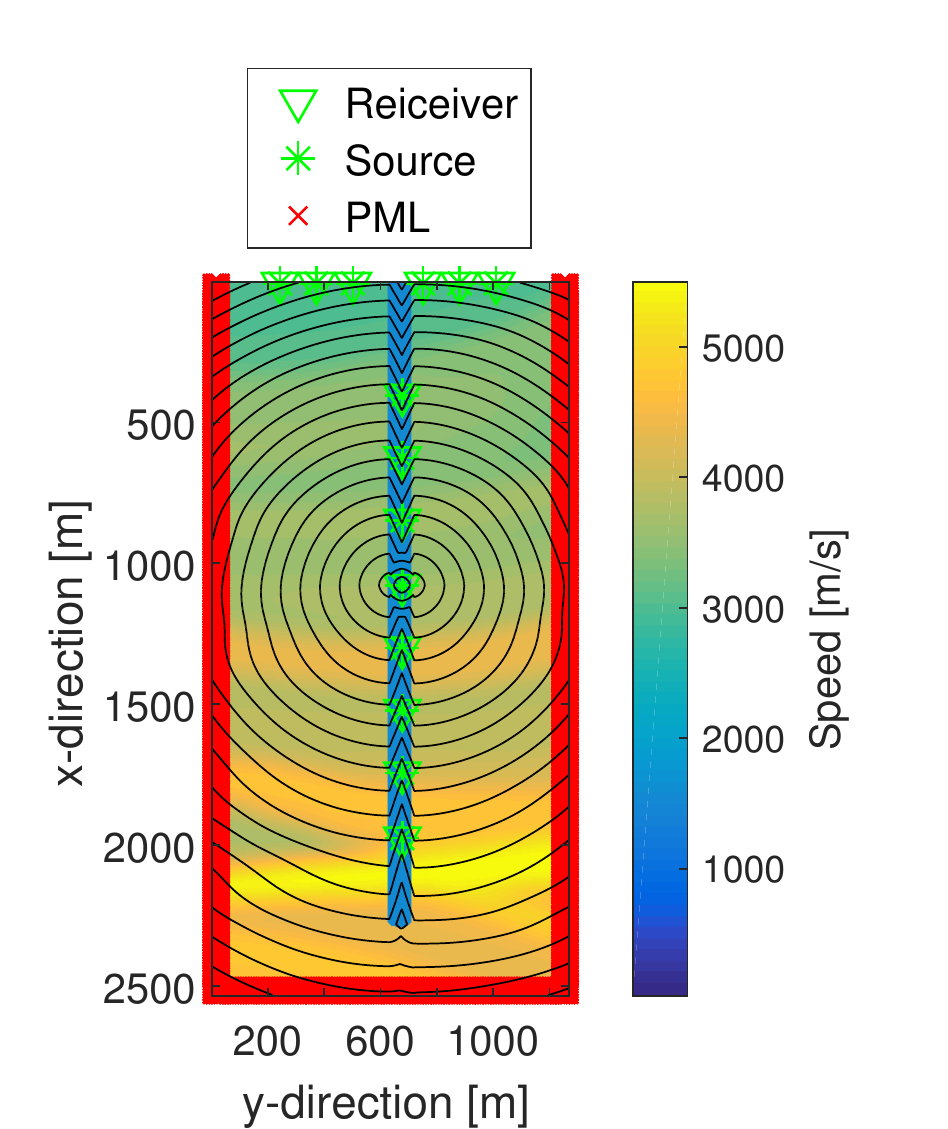}}
~
\subfigure[Isosurfaces of the eikonal time $\bsT_{\rm eik; CM}$ used to approximate the cavity-modes of the solution.]{\label{fig:TeikBorehole2}\includegraphics[trim=0mm 0mm 0mm 0mm, clip,width=0.31\textwidth]{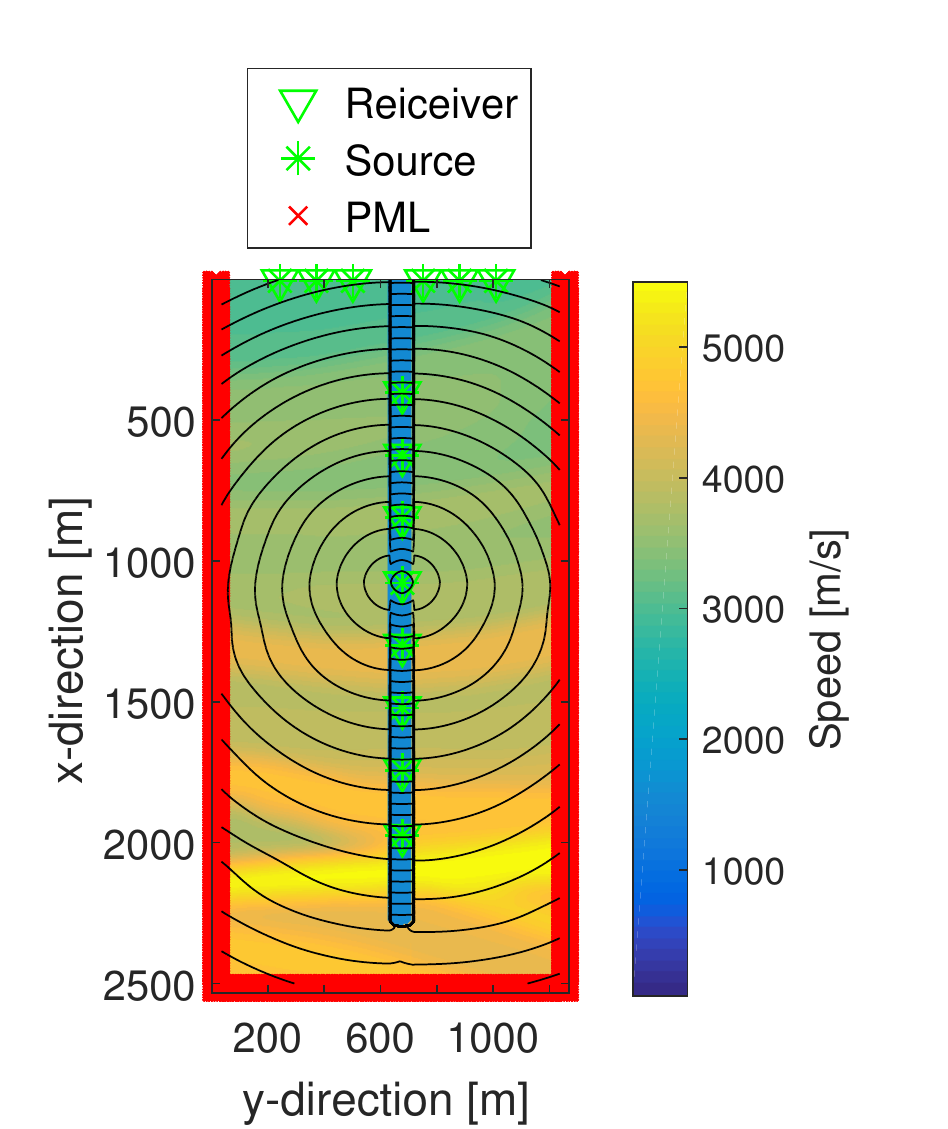}}

~
\subfigure[Time-domain trace of the coinciding source receiver pair number 1 after $m=40$~interpolation points, together with the comparison solution.]{\label{fig:TD12_S1_R1_borhole}\includegraphics [trim=0mm 0mm 0mm 4mm, clip,width=0.9\textwidth]{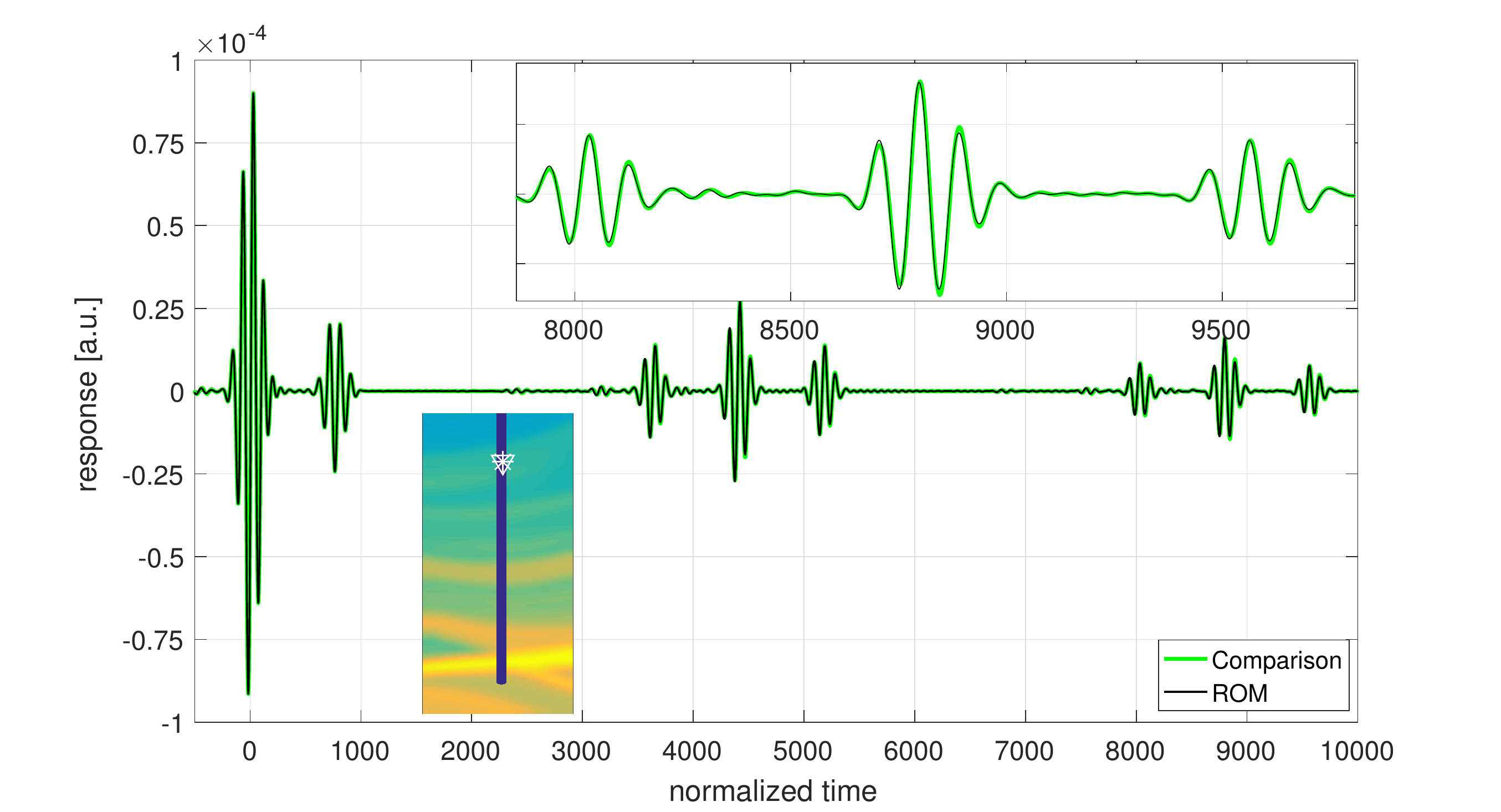}}
~
\subfigure[Time-domain trace from source number 7 inside the borehole to the rightmost surface receiver number 14 after $m=40$~interpolation points.]{\label{fig:TD12_S7_R14_borhole}\includegraphics [trim=0mm 0mm 0mm 4mm, clip,width=0.9\textwidth]{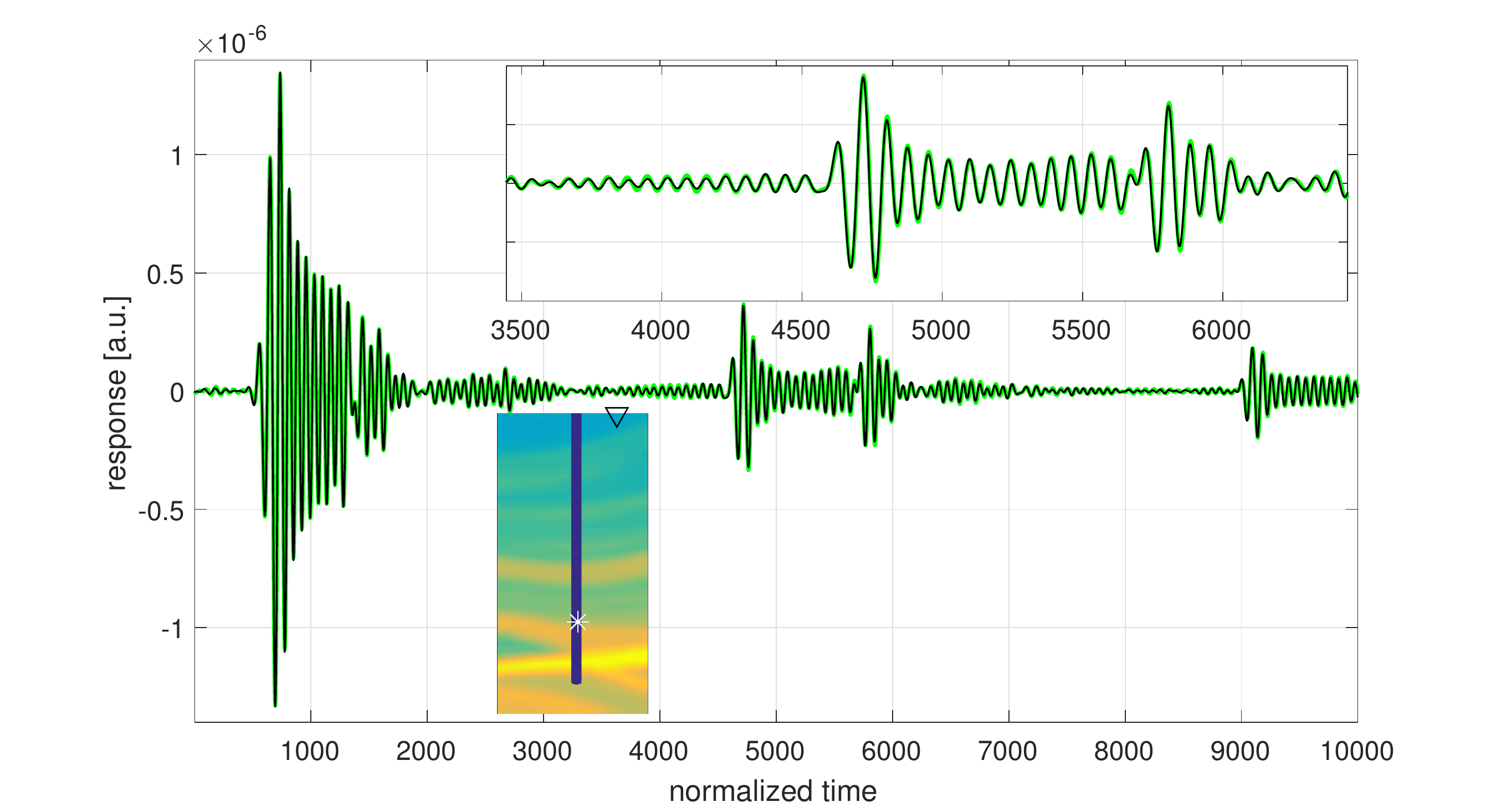}}
\caption{Resonant cavity inside a smooth geology test case.}
\end{figure}

\section{Discussion on parallel implementation}

The numerical experiments of the previous section (using a serial MATLAB prototype code) showed significant compression of large-scale wave propagation due to phase-preconditioning.  
{\color{black}To see how} observed dimensionality reduction can be translated to computational cost reduction using modern high performance platforms, e.g., cloud computing,  we consider the simplest parallel implementation, known in computer science literature as {\color{black}an} 'embarrassingly parallel workflow'.\footnote{Term used for parallelization not requiring horizontal communication between nodes.} 

Like the majority of the projection-based model reduction methods,  the PPRKS can be split into basis construction and ROM evaluation stages, as summarized in  Figure~\ref{fig:OverviewAlgorithm}. This figure is complemented by Table~\ref {tab:CompCost}, where we  compare computational cost estimates for {\color{black}PPRKS with standard RKS} neglecting $O(N_{\rm f})$ terms and  considering  only   parallelism  on the external level.

For both {\color{black}standard RKS and PPRKS} the main cost of the first stage consists of the computation of the block-RKS and {\color{black}the rank-revealing subspace truncation via SVD.}
 Phase-preconditioning adds the negligible cost of solving the eikonal equation   and  the  decomposing the waves into incoming/outgoing amplitudes via (\ref{eq:splitting1} and \ref{eq:splitting2}).
In the table we assume that the block-RKS is computed by assigning solutions of Helmholtz problems for different frequencies and right-hand sides to separate workers,
so that the PPRKS and RKS  require $N_{\rm src} m_{\rm PPRKS}$ and $N_{\rm src} m_{\rm RKS}$ nodes, respectively. The PPRKS obviously reduces the number of the Helmholtz solves; however, in the parallel implementation the most important cost reduction lies in a single solve. In our case, this cost is critical due to the high complexity and poor internal parallel scalability of available Helmholtz solvers. In the table, $\psi(N)$ reflects this (usually faster than logarithmic) growth of the computational complexity of the Helmholtz solver.

Thus, the observed reduction in grid nodes from ${\rm N_{\rm f}}$ to ${\rm N_{\rm c}}$, which lies between one and two orders of magnitude, can result in even  stronger reductions of computation time. The subspace truncation is another poorly-scalable bottle-neck of the basis generation stage (e.g., see \cite{ROM_projectionref}) and the compound effect of the reduction of $m_{\rm PPRKS}$ and ${\rm N_{\rm c}}$ compared  to $m_{\rm RKS}$ and ${\rm N_{\rm f}}$ is more than two orders.
 
The main cost of the second stage is the evaluation of the ROM frequency response at quadrature points in the frequency domain. In particular, the computation of the orthogonal basis and the Galerkin projection are the main bottlenecks with costs that growth linear with respect to the fine grid dimension.
The dimension of the PPRKS approximation space is the product of the size of the compressed amplitude space and the number of sources, which is usually of the same order as the dimension needed for {\color{black}standard} RKS\footnote{{\color{black}Recall that this is due} to the tensor--product structure of the PPRKS approximation space given by  (\ref{eq:finaldecomp})}. Nonetheless, storage of the space is reduced by a factor of $N_{\rm src} N_{\rm f}/N_{\rm c} $ and the computation of the coarse grid amplitudes is obviously cheaper.
However, the phase-preconditioned subspace is frequency dependent, unlike the standard RKS. Therefore, the Galerkin projection should be computed for every frequency for the entire operator $Q(s)$. This is not a significant disadvantage thanks to the possibility of a embarrassingly parallel implementation; for every evaluation frequency a separate worker can be assigned.  Moreover, the compressed tensor-product representation (\ref{eq:finaldecomp})  allows efficient lower level parallelization for the evaluation phase of PPRKS, i.e., column-wise, element-wise and via domain-decomposition of the inner-products. {\color{black}Solving the Galerkin system as well as carrying out inverse Fourier transforms to the time-domain are independent on the grid size and their costs can be neglected.}
\begin{figure}[!htbp]
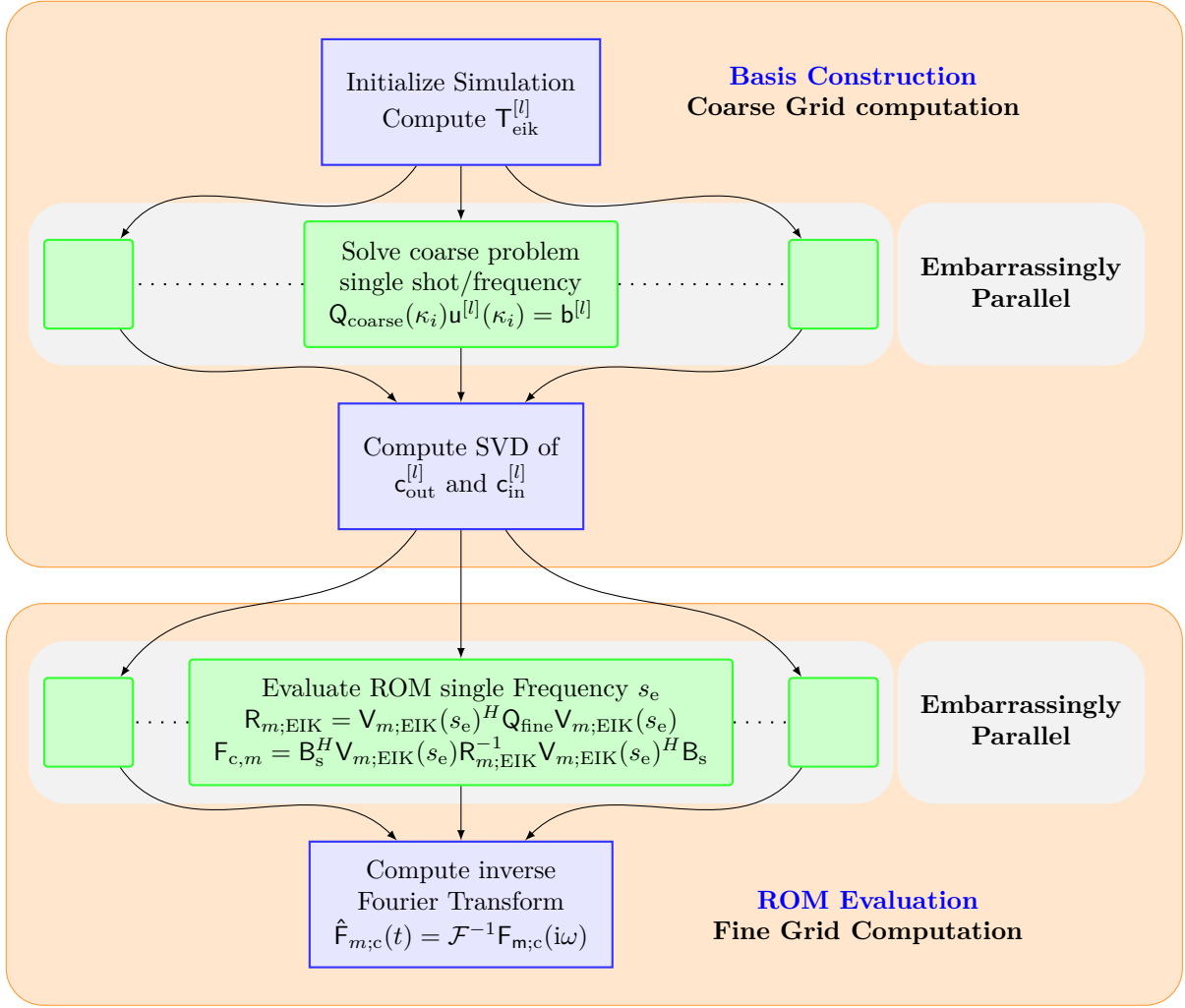

\centering
\include{Images/OverviewAlgorithmDrawing_newBackgr_newText}
\caption{Overview of the proposed algorithm.  External embarrassing parallelism is symbolized by parallel blocks. {\color{black}Internal parallelism within a block is also possible.}}
\label{fig:OverviewAlgorithm}
\end{figure}
%-%-%-%-%-%-%-%-%-%-%-%-%-%-%-%-%-%-%-%-%-%- END DRAWING %-%-%-%-%-%-%-%-%-%-%-%-%-%-%-%-%-%-%-%-%-%-%

\begin{table}[]
\centering
\caption{Cost estimates for  PPRKS and RKS.}
\label{tab:CompCost}
\begin{adjustbox}{max width=\textwidth}
\begin{tabular}{|l||c|c||l l|}
\hline
\multirow{2}{*}{Step}     & \multicolumn{2}{c||}{PPRKS}&\multirow{2}{*}{Legend}&{}\\
\cline{2-3}
{} & Computations per worker & \# workers &{}&{}\\ %\vspace{0.2cm}
\hline
Eikonal & $ O( N_{\rm f}{\rm log}  N_{\rm f})$ 		& $N_{\rm src} $	& $N_\text{src}$ &\#  sources\\
Basis Comp& $ O\left({\rm N_{\rm c}}\psi({\rm N_{\rm c}})\right)$ 	& $N_{\rm src} m_{\rm PPRKS}$ 		&$N_{\rm c/f}$&\# coarse/fine grid notes\\
SVD  &  $O( N_{\rm c} [2 N_{\rm src} m_{PPRKS} ]^2  ) $ &1&$m$&\# ROM interpolation points\\ %ofically min(m n^2, m^2 n) (normal eq.) 
Eval & $ O(N_{\rm f}N_{\rm src}^2 M_{\rm SVD}^2)$ & $N_{\rm eval} $ & ${N_{\rm eval}}$&\# evaluation frequencies\\% O(N_f) because sparse!!
\cline{2-3}
{}&\multicolumn{2}{c||}{RKS}&{$M_{\rm SVD}$}&Size SVD compressed vectors\\
\cline{2-3}
{Basis Comp.}&	 $ O\left({ N_{\rm f}}\psi(N)\right)$	&$N_{\rm src} m_{\rm RKS}$&{}&{}\\
SVD  &  $O( N_{\rm c} [2 N_{\rm src} m_{RKS} ]^2  ) $ &1&$m$&\# ROM interpolation points\\ %ofically min(m n^2, m^2 n) (normal eq.) 
{Eval}&$O(N_{\rm f} N_{\rm src}^2m_{RKS}^2) $ &$  1$&$\psi(N_{\rm c/f})$& scaling function of Helmholtz solver\\% you need to comunicate the ROM for parralel evauation of size(N_s^2 m^2)
\hline
\end{tabular}
\end{adjustbox}
\end{table}

We choose to benchmark the prototype implementation of algorithm in two parts. The basis construction is benchmarked on a CPU and the ROM evaluation on a GPU, as our algorithm is intended for the modern high performance computing environment. Efficient Helmholtz solvers or solvers for large, sparse matrix systems are generally developed for CPUs. GPUs, however, are designed for fast, parallel computation of large inner products and therefore excellent for evaluation stage of the proposed model order reduction technique.

For the smooth geophysical structure example given in this paper with $N_{\rm f}=4\cdot10^5, N_{\rm src}=12$ we compare the significant computation times in Table~\ref{tab:CompCostEval}. The basis computation is performed on a CPU\footnote{Solved using UMFPACK v~5.4.0 on a 4-Core Intel i5-4670 CPU@3.40~GHz with parallel BLAS level-3 routines} and the ROM evaluation on a GPU\footnote{Double precision python implementation on {\color{black}an} Nvidia GTX 1080 Ti}. In the proposed algorithm we solve the wave equation on a coarse grid only, leading to a much lower cost in basis construction than standard RKS were fine systems need to be solved. This is especially important considering that for large 3D applications it can become infeasible to solve the wave equation on a fine grid as the scaling function $\psi(\cdot)$ is much worse for 3D systems than for 2D systems. 
To show the cost of the evaluation of the reduced order model we benchmarked the evaluation of a single frequency $s_e$ on a GPU. The used model has the parameters $N_{\rm f}=4\cdot10^5, M_\text{SVD}=100, N_{\rm src}=12$ and the results are given in {\color{black}Table~\ref{tab:CompCostEval}.} The computationally most involving part is the computation of the Galerkin inner product of left hand vectors $\mathsf{V}_{m;{\rm EIK}}(s_\text{e})^H$ with the vectors $\mathsf{Q}_\text{fine}\mathsf{V}_{m;{\rm EIK}}(s_\text{e})$. Even for this relatively small example the computational cost of the solving a coarse system and projecting the ROM is smaller than evaluating the equation on a fine grid.
{\color{black}We infer that especially for large-scale models these computation times become negligible with respect to basis construction, which scales worse.}
The phase preconditioning approach drastically reduces the vertical communication of the algorithm as only coarse grid amplitudes and phases need to be transferred to all workers instead of fine grid RKS vectors. The ROM is essentially compressed and the storage is drastically reduced.

\begin{table}[]
\centering
\caption{Cost of the basis computation and evaluation of the reduced order model.}
\label{tab:CompCostEval}
\begin{adjustbox}{max width=\textwidth}
\begin{tabular}{|l |c |l |l |}
\hline
Basis Computation comparison & Computation & Time & {}\\
\hline
Block solve fine grid 		& $\bsQ_\text{fine}(s_i)^{-1}\bsB$		& 10.3s & {}\\
Single solve fine grid 	& $\bsQ_\text{fine}(s_i)^{-1}\bsb$		&  4.1s  & {}\\
Block solve coarse grid 	& $\bsQ_\text{coarse}(s_i)^{-1}\bsB$	&  0.6s  & {}\\
Single solve coarse grid 	& $\bsQ_\text{coarse}(s_i)^{-1}\bsb$ 	&  0.2s  & {}\\
\hline
Evaluation Step & Computation & Time & Scaling\\
\hline
Computing phase functions & $\exp{i \omega T_{\rm eik}}$ & 0.00546s & $N_\text{src} N_\text{f}$\\
Hadamard Products & $\exp{i \omega T_{\rm eik}} c_{\rm SVD}$ & 0.01496s & $M_\text{SVD} N_\text{src} N_\text{f}$\\
Galerkin inner product& $\mathsf{V}_{m;{\rm EIK}}(s_\text{e})^H \cdot \mathsf{Q}_\text{fine}\mathsf{V}_{m;{\rm EIK}}(s_\text{e})$ &1.752s & $N_\text{f} M_\text{SVD}^2 N_\text{src}^2$ \\
%Transfer of ROM to CPU & - &5.30115s & $M_\text{SVD}^2 N_\text{src}^2 $\\
%cholestry factorization is available of GPU so we should solve on GPU
%Not sure about banded solver for PML system
\hline
\end{tabular}
\end{adjustbox}
\end{table}

In summary, the computational cost is shifted from the poorly scalable basis construction to the highly scalable evaluation stage where inner products can be computed in an embarrassingly parallel fashion on multiple GPUs. We should also mention significant storage reduction due to phase-preconditioning  as the amplitude basis is smaller than the standard RKS basis for the same accuracy and is stored on the coarse grid only, which significantly reduces vertical communication.

\section{Conclusions}

In this paper we have {\color{black}introduced} phase-preconditioned rational Krylov subspace (PPRKS) for model order reduction and compression of wave propagation in unbounded domains {\color{black}targeting} problems with large propagation distances. {\color{black}Preconditioning} is achieved by splitting the RKS into {\color{black}incoming} and outgoing waves and factoring out {\color{black}strongly-oscillating phase-terms using the WKB approximation.} The remaining slowly-varying amplitude terms are SVD-compressed and then used in the construction of the preconditioned projection space via combinations of the singular vectors of the compressed space and the WKB phase terms computed for different inputs (sources). Finally, the ROM is evaluated via structure-preserving model reduction. 
 
Phase-preconditioning has multiple objectives, namely, reduction of the number of required interpolation points, right-hand sides, and spatial discretization. The number of interpolation points needed for a non-preconditioned RKS method is fundamentally limited by the Nyquist frequency. However, Phase-preconditioning weakens this dependence of the interpolation points on the Nyquist limit. We {\color{black}quantified} this effect for {\color{black}one-dimensional} SISO problems with {\color{black}piecewise} constant coefficients, where the PPRKS solution is exact with the number of the RKS shifts equal to the number of the homogeneous layers, i.e., this number plays the same role as the problem dimensionality in a conventional RKS approach. Thus in 1D the number of interpolation points needed is independent of the Nyquist rate. We do not have a rigorous estimate for the general case of multidimensional MIMO problems. {\color{black}However,} numerical experiments show that the {\color{black}positive effects} of preconditioning can increase due to simultaneous reduction of interpolation points and {\color{black}right-hand} sides. 
In addition, {\color{black}factorization} significantly relaxes requirements on the discretization grids for {\color{black}subspace computations}, which is critical for {\color{black}large-scale} problems due to the poor scalability of {\color{black}available} Helmholtz solvers.
 
Furthermore, factorization reduces the computation cost and increases the model-reduction compression factor. More specifically, for a given approximation accuracy the SVD compressed amplitude space is much smaller than the RKS basis. Numerical experiments for sections of 2D acoustic benchmark Marmousi problem show that the best cost reduction in subspace generation and compression is achieved for smooth wavespeed profiles; however, our approach is still competitive for the discontinuous models and can even be adapted to include resonant substructures.
 
Finally, we point out that due to the tensor product-like structure of the MIMO preconditioned projection space, it is larger than the space of compressed amplitudes and can even be comparable to the conventional block-RKSs required for the same accuracy. 

However, unlike the subspace generation and compression, the projection is generally highly scalable and  can be easily implemented in parallel on GPUs, leaving its computation-time insignificant.
 
In this paper, we presented a prototype implementation of PPRKS for 2D problems using serial computation; however, our eventual target is high-performance computing of large scale 3D seismic problems. In future work, we will also focus on optimal placement of the interpolation points. Specifically, we plan to investigate the approximation quality of the reduced-order models when we move the interpolation points away from the imaginary axis and into the complex plane. This can potentially improve both the approximation properties of the preconditioned RKS for the case of bounded time intervals and the performance of Helmholtz iterative solvers used for RKS construction. As a natural extension of PPRKS, we will also focus on the modeling of wave propagation in dispersive media using PPRKS, since this will not add additional costs to the evaluation stage. Finally, we note that WKB-like asymptotic solutions are available for many discrete and continuous dynamical systems, which opens up a number of possibilities to extend phase-preconditioning to such problems and related matrix-function computations, in particular if the cost of the solution of the shifted systems is dominant in the RKS algorithm.

%\appendix
%\section{An example appendix} 
%\section*{Acknowledgments}

%todo:abreviate journalnames in references?
\bibliographystyle{siamplain}
\bibliography{references}
\end{document}

%% file: preamble.tex
\def\be#1\ee{\begin{equation}#1\end{equation}}
\newcommand{\bea}{\begin{eqnarray}}
\newcommand{\eea}{\end{eqnarray}}
\newcommand{\beas}{\begin{eqnarray*}}
\newcommand{\eeas}{\end{eqnarray*}}

\newcommand{\calK}{\mathcal{K}}

\renewcommand{\Re}{\mathfrak{Re} \,}
\renewcommand{\Im}{\mathfrak{Im} \,}

\newcommand{\Kb}[2]{ \mathcal{K}_{#1} \left(  #2 \right) }

\renewcommand{\exp}[1]{\ensuremath{\,\mathrm{exp}\!\left(#1\right)}}
\newsavebox{\measurebox}
\def\bsA{\mathsf{A}}
\def\bsB{\mathsf{B}}

\def\bsD{\mathsf{D}}

\def\bsQ{\mathsf{O}}

\def\bsQ{\mathsf{Q}}
\def\bsR{\mathsf{R}}

\def\bsT{\mathsf{T}}
\def\bsU{\mathsf{U}}

\def\bsb{\mathsf{b}}

\def\bsp{\mathsf{p}}

\def\bsu{\mathsf{u}}

\def\bsx{\mathsf{x}}
\def\bsy{\mathsf{y}}
\def\bsz{\mathsf{z}}
\newcommand{\RR}{{\mathbb{R}}}
\newcommand{\CC}{{\mathbb{C}}}

%% file: Images/OverviewAlgorithmDrawing_newBackgr_newText.tex
% - Define box types - %
\tikzstyle{central}=[rectangle,
                                    thick,
                                    minimum size=1.7cm,
                                    draw=blue!80,
                                    fill=blue!10]

\tikzstyle{parralell}=[rectangle,
                                    thick,
                                    minimum size=1.2cm,
                                    draw=green!80,
                                    fill=green!20,
                                    rounded corners=0.5mm]
                                                                           
\tikzstyle{parralelltext}=[rectangle,
                                    thick,
                                    minimum size=1.7cm,
                                    draw=green!80,
                                    fill=green!20,
                                    rounded corners=0.5mm]

\tikzstyle{background}=[rectangle,
                                                fill=gray!10,
                                                inner sep=0.2cm,
                                                minimum size=2.2cm,
                                                minimum width=3.2cm,
                                                rounded corners=5mm]

\tikzstyle{background2}=[rectangle,
			       draw=orange!80,
                                                fill=orange!20,
                                                inner sep=0.5cm,
                                                outer sep=1cm,
                                                minimum size=2.2cm,
                                                minimum width=3.2cm,
                                                rounded corners=5mm]

\tikzstyle{explain}=[rectangle,
                                                fill=gray!10,
                                                inner sep=0.1cm,
                                                minimum size=2.2cm,
                                                minimum width=3.0cm,
                                                rounded corners=5mm]

% - % - % - % - begin figure - % - % - % - % - %
\begin{tikzpicture}[>=latex,text height=1.5ex,text depth=0.25ex]
  \matrix[row sep=0.5cm,column sep=0.25 cm] {
  
% - draw boxes - %
%central node1
&  
& 
&\node(c1)[central]{\begin{tabular}{c} Initialize Simulation \\ Compute $\mathsf{T}_{\rm eik}^{[l]}$  \end{tabular}};
& 
&
&
&\\
  
%frist parralel node
    &\node(p1l)[parralell]{};
  &\node(p1temp1){};
  &\node(p1c)[parralelltext]{  \begin{tabular}{c} Solve coarse problem \\ single shot/frequency \\ $\mathsf{Q}_\text{coarse}(\kappa_i)\mathsf{u}^{[l]}(\kappa_i)=\mathsf{b}^{[l]}$\end{tabular}};
  &\node(p1temp2){};
  &\node(p1r)[parralell]{}; % \node(ex1)[explain]{ \begin{tabular}{c} {\bf Embarrassingly}\\ {\bf Parallel}\end{tabular}}; \\
  & \node(p1ex1)[explain]{ \begin{tabular}{c} {\bf Embarrassingly}\\ {\bf Parallel}\end{tabular}};\\
  
 %central node2
&\node(dummy1){}; 
&\node(dummy2){};
&\node(c2)[central]{\begin{tabular}{c} Compute SVD of \\ $\mathsf{c}^{[l]}_\text{out}$ and $\mathsf{c}^{[l]}_\text{in}$  \end{tabular}};
& 
&\\
  
&  
& 
&\node(cdummy){};
& 
&
&\\

%second paralell node
  &\node(p2l)[parralell]{};
   &\node(p2temp1){};
  &\node(p2c)[parralelltext]{  \begin{tabular}{c} Evaluate ROM  single Frequency $s_\text{e}$ \\ $\mathsf{R}_{m;\text{EIK}}=\mathsf{V}_{m;{\rm EIK}}(s_\text{e})^H  \mathsf{Q}_\text{fine}\mathsf{V}_{m;{\rm EIK}}(s_\text{e})$ \\ $\mathsf{F}_{\text{c},m}=\mathsf{B}_\text{s}^H\mathsf{V}_{m;{\rm EIK}}(s_\text{e}) \mathsf{R}_{m;\text{EIK}}^{-1}  \mathsf{V}_{m;{\rm EIK}}(s_\text{e})^H\mathsf{B}_\text{s}$\end{tabular}};
  &\node(p2temp2){};
  &\node(p2r)[parralell]{};
  &\node(p2ex1)[explain]{ \begin{tabular}{c} {\bf Embarrassingly}\\ {\bf Parallel}\end{tabular}}; \\

%cetral node 3
&
&
&\node(c3)[central]{\begin{tabular}{c} Compute inverse\\ {Fourier Transform} \\ $\mathsf{\hat{F}}_{m;\text{c}}(t)=\mathcal{F}^{-1} \mathsf{\mathsf{F}_{m;\text{c}}(\text{i} \omega) }$  \end{tabular}};
&
&
&\\%\node(dummy3){\begin{tabular}{c} {\bf Evaluation Phase:} \\ Embarrassingly parallel:\\ Evaluation frequencies \\ \& Subdomains \end{tabular}};  \\
    };

% - draw arrows - %

\draw[black,->]         (c1) to[out=235,in=55]   (p1l);	
\draw[black,->]         (c1) --  (p1c);
\draw[black,->]         (c1) to[out=305,in=125]  (p1r);

%\draw[black,->]         (p1l) to [out=305,in=135] node[midway,above] {Communication} node[midway,below] {Cost}    (c2);
\draw[black,->]         (p1l) to [out=305,in=135]    (c2);
\draw[black,->]         (p1c) --   (c2);
\draw[black,->]         (p1r) to[out=235,in=45]   (c2);
         
\draw[black,->]         (c2) to[out=235,in=55]  (p2l);	
\draw[black,->]         (c2) -- (p2c);
\draw[black,->]         (c2) to[out=305,in=125]  (p2r);
                
\draw[black,->]         (p2l)  to [out=305,in=135]   (c3);
\draw[black,->]         (p2c) --  (c3);
\draw[black,->]       (p2r)   to[out=235,in=45] (c3);

% - draw dotted lines - %                                                                                                                                                                                                                                                                        
\draw[black, loosely dotted, thick] (p1c) -- (p1r);
\draw[black, loosely dotted, thick] (p1c) -- (p1l);
\draw[black, loosely dotted, thick] (p2c) -- (p2r);
\draw[black, loosely dotted, thick] (p2c) -- (p2l);

% - draw backgrounds - %                                                                                                                                                                                                                                                                        
    \begin{pgfonlayer}{background}

     \node [background2,
                  fit=(p2l) (c3) (p2ex1),
 %                  label={[xshift=4.6cm, yshift=-6.5cm]  \begin{tabular}{c}  {\emph{Parallelization}:} \\{Embarrassingly parallel:} \\{External: evaluation frequencies}\\{\& Internal: subdomains/RHS} \end{tabular} } 
       %             label={[xshift=0cm, yshift=-6.5cm]   \begin{tabular}{c}  {\bf Evaluation Phase}\\ {\bf Fourier Transform} \\{Embarrassingly parallel:} \\{External: evaluation frequencies}\\{\& Internal: subdomains/RHS} \end{tabular} } 
 %                  label=south east: {\bf Embarrassingly  Parallel} 
                   ] {};        
                   
\node [background2,
                  fit=(p1l) (c1) (p1ex1)(c2),
                  label={[xshift=3.5cm, yshift=-2.5cm]   \begin{tabular}{c}  
                  {\color{blue}\bf Basis Construction}\\ {\bf Coarse Grid computation} \end{tabular} } 
 %                  label={[xshift=-4.5cm, yshift=-3.5cm]   \begin{tabular}{c}  {\color{blue}\bf Coarse Grid Computation}\\ {\bf Construction phase \&}\\{\bf SVD Compression} \end{tabular} } 
 %                label={[xshift=10cm, yshift=13cm]  \begin{tabular}{l} {\# Problems:}\\ {\# Evaluation}\\ { frequencies} \end{tabular}  }
 %                  label=south east: {\bf Embarrassingly  Parallel} 
                   ] {};

     \node [background,
                  fit=(p1l) (p1c) (p1r),
%                   label={[xshift=4.5cm, yshift=2cm]  \begin{tabular}{c}  {\emph{Parallelization}:}\\ {External (embarrassingly):}\\ {Interpolation frequencies}\\{ Internal: Solver parallelization} \end{tabular}  } 
                   %\begin{tabular}{c}{ {\bf Evaluation Phase:} \\ Embarrassingly parallel:\\  \\ \& Subdomains} \end{tabular}} 
                   % \begin{tabular}{c} {   \# Problems:}\\{$N_\text{src} \cdot m$} \end{tabular} }
 %                  label=south west: {  \# of problems:} 
                   ] {};
                   
   %  \node [background,
     %             fit=(dummy1)(dummy2),
      %             label=center:\begin{tabular}{c} {\bf Embarrassingly}\\ {\bf Parallel}\end{tabular}
        %           ] {};

     \node [background,
                  fit=(p2l) (p2c) (p2r),
                   label={[xshift=5.5cm, yshift=-4cm]  \begin{tabular}{c} {\color{blue}{\bf ROM Evaluation}} \\ {\bf Fine Grid Computation} \end{tabular} }
    %            label=south east: {\bf Embarrassingly  Parallel} 
                   ] {};

    \end{pgfonlayer}
\end{tikzpicture}

%% file: SJSC_Draft_V6.bbl
\begin{thebibliography}{10}

\bibitem{InterpolatoryModelReduction}
{\sc A.~Antoulas, C.~Beattie, and S.~Gugercin}, {\em {Interpolatory Model
  Reduction of Large-Scale Dynamical Systems}}, in Efficient Modeling and
  Control of Large-Scale Systems, J.~Mohammadpour and K.~Grigoriadis, eds.,
  Springer US, Boston, MA, 2010, ch.~1, pp.~3--58,
  \url{https://doi.org/10.1007/978-1-4419-5757-3}.

\bibitem{beckermann2009error}
{\sc B.~Beckermann and L.~Reichel}, {\em {Error estimates and evaluation of
  matrix functions via the Faber transform}}, SIAM Journal on Numerical
  Analysis, 47 (2009), pp.~3849--3883.

\bibitem{Bender&Orszag}
{\sc C.~M. Bender and S.~A. Orszag}, {\em Advanced Mathematical Methods for
  Scientists and Engineers}, Springer-Verlag New York, 1999.

\bibitem{ParametricMOR}
{\sc P.~Benner, S.~Gugercin, and K.~Willcox}, {\em {A Survey of
  Projection-Based Model Reduction Methods for Parametric Dynamical Systems}},
  SIAM Review, 57 (2015), pp.~483--531,
  \url{https://doi.org/10.1137/130932715},
  \url{http://dx.doi.org/10.1137/130932715},
  \url{https://arxiv.org/abs/http://dx.doi.org/10.1137/130932715}.

\bibitem{Bindel2006}
{\sc D.~S. Bindel, Z.~Bai, and J.~W. Demmel}, {\em {Model Reduction for RF MEMS
  Simulation}}, Springer Berlin Heidelberg, Berlin, Heidelberg, 2006,
  pp.~286--295, \url{https://doi.org/10.1007/11558958_34},
  \url{http://dx.doi.org/10.1007/11558958_34}.

\bibitem{bourgeois1991marmousi}
{\sc A.~Bourgeois, M.~Bourget, P.~Lailly, M.~Poulet, P.~Ricarte, and
  R.~Versteeg}, {\em {Marmousi, model and data}}, The Marmousi Experience,
  European Association of Exploration Geophysicists,  (1991), pp.~5--16.

\bibitem{MimeticChew}
{\sc W.~C. Chew}, {\em {Electromagnetic theory on a lattice}}, Journal of
  Applied Physics, 75 (1994), pp.~4843 -- 4850,
  \url{https://doi.org/10.1063/1.355770}.

\bibitem{chung2014generalized}
{\sc E.~T. Chung, Y.~Efendiev, and W.~T. Leung}, {\em {Generalized Multiscale
  Finite Element Methods for Wave Propagation in Heterogeneous Media}},
  Multiscale Modeling \& Simulation, 12 (2014), pp.~1691--1721,
  \url{https://doi.org/10.1137/130926675},
  \url{https://doi.org/10.1137/130926675},
  \url{https://arxiv.org/abs/https://doi.org/10.1137/130926675}.

\bibitem{PML2016SiamReview}
{\sc V.~Druskin, S.~{G\"{u}ttel}, and L.~Knizhnerman}, {\em {Near-Optimal
  Perfectly Matched Layers for Indefinite Helmholtz Problems}}, SIAM Review, 58
  (2016), pp.~90--116, \url{https://doi.org/10.1137/140966927},
  \url{http://dx.doi.org/10.1137/140966927},
  \url{https://arxiv.org/abs/http://dx.doi.org/10.1137/140966927}.

\bibitem{druskin2016multi}
{\sc V.~Druskin, A.~V. Mamonov, and M.~Zaslavsky}, {\em {Multiscale S-Fraction
  Reduced-Order Models for Massive Wavefield Simulations}}, Multiscale Modeling
  \& Simulation, 15 (2017), pp.~445--475,
  \url{https://doi.org/10.1137/16M1072103},
  \url{https://doi.org/10.1137/16M1072103},
  \url{https://arxiv.org/abs/https://doi.org/10.1137/16M1072103}.

\bibitem{Remis2013PKS}
{\sc V.~Druskin and R.~Remis}, {\em {A Krylov Stability-Corrected
  Coordinate-Stretching Method to Simulate Wave Propagation in Unbounded
  Domains}}, SIAM Journal on Scientific Computing, 35 (2013), pp.~B376--B400,
  \url{https://doi.org/10.1137/12087356X},
  \url{http://dx.doi.org/10.1137/12087356X},
  \url{https://arxiv.org/abs/http://dx.doi.org/10.1137/12087356X}.

\bibitem{Remis2014EKS}
{\sc V.~Druskin, R.~Remis, and M.~Zaslavsky}, {\em {An extended Krylov subspace
  model-order reduction technique to simulate wave propagation in unbounded
  domains}}, Journal of Computational Physics, 272 (2014), pp.~608 -- 618,
  \url{https://doi.org/http://dx.doi.org/10.1016/j.jcp.2014.04.051},
  \url{http://www.sciencedirect.com/science/article/pii/S0021999114003271}.

\bibitem{Druskin20123883}
{\sc V.~Druskin and M.~Zaslavsky}, {\em {On convergence of Krylov subspace
  approximations of time-invariant self-adjoint dynamical systems}}, Linear
  Algebra and its Applications, 436 (2012), pp.~3883 -- 3903,
  \url{https://doi.org/http://dx.doi.org/10.1016/j.laa.2011.02.039},
  \url{//www.sciencedirect.com/science/article/pii/S002437951100173X}.
\newblock Special Issue dedicated to Heinrich Voss's 65th birthday.

\bibitem{engquist2011sweeping}
{\sc B.~Engquist and L.~Ying}, {\em {Sweeping preconditioner for the Helmholtz
  equation: moving perfectly matched layers}}, Multiscale Modeling \&
  Simulation, 9 (2011), pp.~686--710.

\bibitem{gockler2013convergence}
{\sc T.~G{\"o}ckler and V.~Grimm}, {\em {Convergence Analysis of an Extended
  Krylov Subspace Method for the Approximation of Operator Functions in
  Exponential Integrators}}, SIAM Journal on Numerical Analysis, 51 (2013),
  pp.~2189--2213.

\bibitem{guttel2010rational}
{\sc S.~G{\"u}ttel}, {\em {Rational Krylov methods for operator functions}},
  PhD thesis, Institut f{\"u}r Numerische Mathematik und Optimierung,
  Technische Universit{\"a}t Bergakademie Freiberg, 2010.

\bibitem{Haber20114403}
{\sc E.~Haber and S.~MacLachlan}, {\em {A fast method for the solution of the
  Helmholtz equation}}, Journal of Computational Physics, 230 (2011), pp.~4403
  -- 4418, \url{https://doi.org/http://doi.org/10.1016/j.jcp.2011.01.015},
  \url{http://www.sciencedirect.com/science/article/pii/S0021999111000337}.

\bibitem{barnett1983matrix}
{\sc {I. Gohberg and P. Lancaster and L. Rodman}}, {\em {Matrix Polynomials}},
  SIAM, Philadelphia, PA 19104-2688 USA, 2009.

\bibitem{iserles2006highly}
{\sc A.~Iserles, S.~N{\o}rsett, and S.~Olver}, {\em {Highly oscillatory
  quadrature: The story so far}}, in Numerical mathematics and advanced
  applications, Springer, 2006, pp.~97--118.

\bibitem{DruskinDiffusion}
{\sc L.~Knizhnerman, V.~Druskin, and M.~Zaslavsky}, {\em {On Optimal
  Convergence Rate of the Rational Krylov Subspace Reduction for
  Electromagnetic Problems in Unbounded Domains}}, {SIAM Journal on Numerical
  Analysis}, 47 (2009), pp.~953--971, \url{https://doi.org/10.1137/080715159},
  \url{http://dx.doi.org/10.1137/080715159},
  \url{https://arxiv.org/abs/http://dx.doi.org/10.1137/080715159}.

\bibitem{li2015phase}
{\sc Y.~E. Li and L.~Demanet}, {\em {Phase and amplitude tracking for seismic
  event separation}}, Geophysics, 80 (2015), pp.~WD59--WD72.

\bibitem{Mulder1}
{\sc W.~A. Mulder and R.-E. Plessix}, {\em {Time- versus frequency-domain
  modelling of seismic wave propagation}}, EAGE, 64th Annual Conference and
  Exhibition, Extended Abstract E015,  (2002).

\bibitem{ROM_projectionref}
{\sc M.~O'Connell, M.~E. Kilmer, E.~de~Sturler, and S.~Gugercin}, {\em
  {Computing Reduced Order Models via Inner-Outer Krylov Recycling in Diffuse
  Optical Tomography}}, SIAM Journal on Scientific Computing, 39 (2017),
  pp.~B272--B297, \url{https://doi.org/10.1137/16M1062880},
  \url{http://dx.doi.org/10.1137/16M1062880},
  \url{https://arxiv.org/abs/http://dx.doi.org/10.1137/16M1062880}.

\bibitem{Ruhe1984391}
{\sc A.~Ruhe}, {\em {Rational Krylov sequence methods for eigenvalue
  computation}}, Linear Algebra and its Applications, 58 (1984), pp.~391 --
  405, \url{https://doi.org/http://dx.doi.org/10.1016/0024-3795(84)90221-0},
  \url{//www.sciencedirect.com/science/article/pii/0024379584902210}.

\bibitem{CyclicReference}
{\sc A.~A. Samarskij and E.~S. Nikolaev}, {\em Numerical Methods for Grid
  Equations, Volume I Direct Methods}, Birkhäuser Basel, 1989.

\bibitem{SethianFFM}
{\sc J.~A. Sethian}, {\em { A fast marching level set method for monotonically
  advancing fronts}}, Proceedings of the National Academy of Sciences of the
  United States of America, 93 (1996), pp.~1591--159.

\bibitem{Townsend20140585}
{\sc A.~Townsend and L.~N. Trefethen}, {\em Continuous analogues of matrix
  factorizations}, Proceedings of the Royal Society of London A: Mathematical,
  Physical and Engineering Sciences, 471 (2014),
  \url{https://doi.org/10.1098/rspa.2014.0585},
  \url{http://rspa.royalsocietypublishing.org/content/471/2173/20140585},
  \url{https://arxiv.org/abs/http://rspa.royalsocietypublishing.org/content/471/2173/20140585.full.pdf}.

\bibitem{Weideman2007}
{\sc J.~A.~C. Weideman and L.~N. Trefethen}, {\em {Parabolic and Hyperbolic
  Contours for Computing the Bromwich Integral}}, Mathematics of Computation,
  76 (2007), pp.~1341--1356.

\end{thebibliography}
